\documentclass[11pt]{amsart}

\usepackage{amsmath, amssymb, amsthm, amsfonts, graphicx, pinlabel}

\input{Sections/Preamble}

\begin{document}

\title[Equivalence classes of Augmentations and MCSs]{Equivalence Classes of Augmentations and Morse Complex Sequences of Legendrian Knots}

\author{Michael B. Henry}
\address{Siena College, Loudonville, NY 12211}
\email{mbhenry@siena.edu}

\author{Dan Rutherford}
\address{Ball State University, Muncie, IN 47306}
\email{rutherford@bsu.edu}

\begin{abstract}
Let $\Leg$ be a Legendrian knot in $\R^3$ with the standard contact structure.
In \cite{Henry2011}, a map was constructed from equivalence classes of Morse complex sequences for $\Leg$, which are combinatorial objects motivated by generating families, to homotopy classes of augmentations of the Legendrian contact homology algebra of $\Leg$.  Moreover, this map was shown to be a surjection.  We show that this correspondence is, in fact, a bijection.  As a corollary, homotopic augmentations determine the same graded normal ruling of $\Leg$ and have isomorphic linearized contact homology groups. A second corollary states that the count of equivalence classes of Morse complex sequences of a Legendrian knot is a Legendrian isotopy invariant.
\end{abstract}

\maketitle

\section{Introduction}
\mylabel{s:intro}

The symplectic techniques of holomorphic curves and generating families provide two effective classes of invariants of Legendrian knots in standard contact $\R^3$. The holomorphic curve approach, which in this low-dimensional setting takes on a combinatorial flavor, can be used to define a Differential Graded Algebra (DGA), known alternatively as the Legendrian contact homology DGA or the Chekanov-Eliashberg DGA and originally defined in \cite{Chekanov2002a} and \cite{Eliashberg2000}. Generating families of Legendrian submanifolds in $1$-jet spaces, including $\R^3$, have also been used to produce  homological Legendrian invariants; see, for instance, \cite{Jordan2006}, \cite{SabloffT2013}, \cite{Traynor1997}, \cite{Traynor2001}.  In addition to distinguishing Legendrian isotopy classes of knots, both the holomorphic and generating family invariants carry useful information about Lagrangian cobordisms, cf. \cite{EkholmHK}, \cite{SabloffT2013}. 

For Legendrian knots in $\R^3$, several close connections have been discovered between holomorphic curve and generating family invariants, although many questions remain.  For example, the existence of a linear at infinity generating family for a Legendrian knot is known to be equivalent to the existence of a certain DGA morphism, called an augmentation, from the Chekanov-Eliashberg DGA to its ground ring, (\cite{Chekanov2005}, \cite{Fuchs2003}, \cite{Fuchs2004}, \cite{Fuchs2008}, \cite{Sabloff2005}). However, it is unknown if this statement can be strengthened to a bijective correspondence between appropriate equivalence classes of generating families and augmentations. In this article, we approach this question using a discrete analog of a generating family called a Morse complex sequence, abbreviated MCS. MCSs have proven to be more tractable for explicit construction and computation; see, for example, \cite{Henry2011,Henry2013,Henry2014}.

A generating family for a Legendrian $\Leg \subset \R^3$ is a one-parameter family of functions whose critical values coincide with the projection of $\Leg$ to the $xz$-plane, also called the front diagram of $\Leg$.  A Morse complex sequence for $\Leg$ is a collection of chain complexes and formal handleslide marks on the front diagram of $\Leg$ that obey a set of restrictions identical to those satisfied by Morse complexes in a generic one-parameter family of functions.  Thus, by considering MCSs rather than generating families, a family of functions is replaced by discrete, algebraic data.  There is a natural  equivalence relation on the set of MCSs on $\Leg$ that reflects the generic bifurcations appearing in two-parameter families of Morse complexes. 

The concept of a Morse complex sequence originally appeared in unpublished work of Petya Pushkar, and first appears in print in the work of the first author \cite{Henry2011} where MCSs are studied in connection with augmentations.  In \cite{Henry2011}, a surjective map is defined from MCSs of $\Leg$ to augmentations of the Chekanov-Eliashberg DGA of $\Leg$. Moreover, equivalent MCSs are mapped to homotopic augmentations. In the present article, we complement the results of \cite{Henry2011} by showing in Lemma~\ref{l:equiv} that two MCSs mapped to homotopic augmentations must, in fact, be equivalent as MCSs. Combined with \cite{Henry2011} this gives the following.

\begin{theorem} \label{t:bijection}
For any Legendrian knot $\Leg \subset \R^3$ with generic front diagram, there is a bijection between equivalence classes of Morse complex sequences for $\Leg$ and homotopy classes of augmentations of the Chekanov-Eliashberg DGA of $\Leg$.
\end{theorem}

As a consequence,  the number of MCS equivalence classes is a Legendrian isotopy invariant; see Corollary \ref{c:MCS-invt}.  The less immediate Corollary \ref{c:LCH} combines Theorem \ref{t:bijection} with previous work of the authors from \cite{Henry2013} to deduce that homotopic augmentations must have isomorphic linearized homology groups.  The set of linearized homology groups is a Legendrian isotopy invariant. Corollary \ref{c:LCH} allows for a refinement of this invariant by considering multiplicities. 

The remainder of the article is organized as follows.  Section \ref{s:background} recalls background concerning augmentations and Morse complex sequences, while Section \ref{s:results} contains the proof of Theorem \ref{t:bijection} and its corollaries.

\subsection{Acknowledgments}
We gratefully acknowledge The Royal Academy of Belgium where conversations between the authors on this project began in August 2013 at a workshop on Legendrian submanifolds, holomorphic curves, and generating families. We also thank the workshop organizer Fr\'{e}d\'{e}ric Bourgeois.  In addition, we thank the American Institute of Mathematics for supporting a SQuaRE research group on augmentations and related topics, and we thank our fellow SQuaRE participants Dmitry Fuchs, Paul Melvin, Josh Sabloff, and Lisa Traynor.    
\section{Background}
\mylabel{s:background}

A \textbf{Legendrian knot} in the standard contact structure on $\R^3$ is a smooth knot $\Leg : S^1 \to \R^3$ satisfying $\Leg'(t) \in \ker(dz-y\,dx)$ for all $t \in S^1$. A smooth one-parameter family $\Leg_t$, $0 \leq t \leq 1$, of Legendrian knots is a \textbf{Legendrian isotopy} between $\Leg_0$ and $\Leg_1$. The \textbf{front diagram} of $\Leg$ is the projection of $\Leg$ to the $xz$-plane. Every Legendrian knot is Legendrian isotopic, by an arbitrarily small Legendrian isotopy, to a Legendrian knot whose front diagram is embedded except at transverse self-intersections, called \textbf{crossings}, and semi-cubical cusps such that, in addition, all of these exceptional points have distinct $x$-coordinates. A Legendrian knot with such a front diagram is said to have a \textbf{$\sigma$-generic} front diagram; see, for example, the front diagram in Figure~\ref{f:knot}.  In a neighborhood of an $x$ value that is not the $x$-coordinate of a crossing or cusp, the front diagram looks like a collection of non-intersecting line segments commonly called the \textbf{strands} of $\front$ at $x$. Orient $\Leg$. The \textbf{rotation number} $r(\Leg)$ is $(d - u)/2$ where $d$ (resp. $u$) is the number of cusps at which the orientation travels downward (resp. upward) with respect to the $z$-axis.

\begin{figure}[t]
\centering
\includegraphics[scale=.9]{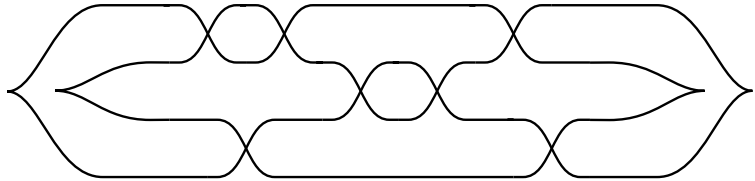}
\caption{A $\sigma$-generic front diagram of a Legendrian knot with rotation number $0$.}
\label{f:knot}
\end{figure}

\subsection{Chekanov-Eliashberg Algebra}

Fix a Legendrian knot $\Leg$ with $\sigma$-generic front diagram $\front$ and rotation number $0$. A \textbf{Maslov potential} is a map $\Maslov : \Leg \to \Z$ that is constant except at cusp points of $\Leg$ where the Maslov potential of the lower strand of the cusp is 1 less than the upper strand. Let $A(\front)$ be the $\Z / 2 \Z$ vector space generated by the labels $Q = \{q_1, \hdots, q_n\}$ assigned to the crossings and right cusps of $\front$. A generator $q \in Q$ is assigned a \textbf{grading} $|q|$, also called a \textbf{degree}, so that $|q|$ is 1 if $q$ is a right cusp and, otherwise, $|q|$ is $\Maslov(T) - \Maslov(B)$ where $T$ and $B$ are the strands of $\front$ crossing at $q$ and $T$ has smaller slope. The graded algebra $\alg(\front)$ is the unital tensor algebra $TA(\front)$.  The \textbf{Chekanov-Eliashberg algebra}, written $(\alg(\front), \df)$, is the algebra $\alg(\front)$ along with a degree -1 differential $\df : \alg(\front) \to \alg(\front)$ that, in the case of the front diagram description from \cite{Ng2003}, is defined by counting certain admissible maps of the two-disk $D^2$ into the $xz$-plane. We refer the reader to \cite{Ng2003} for a careful definition of $\df$ as we will need to investigate only a small subset of these maps. 

An \textbf{augmentation} is a map $\aug : \alg(\front) \to \Z / 2 \Z$ satisfying $\aug \circ \df = 0$, $\aug(1) = 1$, and $\aug(q) = 1$ only if $|q|=0$. The set $\mbox{Aug}(\front)$ is the set of all augmentations of $(\alg(\front), \df)$. We say a crossing $q$ is \textbf{augmented} by $\aug$ if $\aug(q) = 1$. An augmentation can be thought of as a morphism between the differential graded algebra $(\alg(\front), \df)$ and the differential graded algebra $(\Z / 2 \Z, \df')$ whose only non-zero element is in degree $0$ and where $\df' = 0$. From this perspective, there is a natural algebraic equivalence relation on $\mbox{Aug}(\front)$. Given $\aug$ and $\aug'$ in $\mbox{Aug}(\front)$, a \textbf{chain homotopy} from $\aug$ to $\aug'$ is a degree $1$ linear map $H : (\alg(\front), \df) \to (\Z / 2 \Z, \df')$ satisfying $\aug - \aug' =  \df' \circ H +H \circ \df $ and $H(ab) = H(a) \aug'(b) + (-1)^{|a|}\aug(a) H(b)$ for all $a, b \in \alg(\front)$. Since we are working over $\Z / 2 \Z$ and $\df'=0$, these conditions simplify to
\begin{equation}
\label{eq:chain-homotopy}
\aug - \aug' = H \circ \df \mbox{ and } H(ab) = H(a)\aug'(b) + \aug(a)H(b).
\end{equation}
By Lemma 2.18 of \cite{Kalman2005}, a chain homotopy $H$ is determined by the values it takes on the degree $-1$ crossings of $\front$. 

We say augmentations $\aug$ and $\aug'$ are \textbf{homotopic} and write $\aug \simeq \aug'$ if there exists a chain homotopy from $\aug$ to $\aug'$. As the notation implies and as is proven in \cite{YvesFelix1995}, chain homotopy provides an equivalence relation on the set $\mbox{Aug}(\front)$. We let $\mbox{Aug}^{ch}(\front)$ be $\mbox{Aug}(\front) / \simeq$. By Proposition 4.5 of \cite{Henry2011}, the count of homotopy classes of augmentations is a Legendrian isotopy invariant. 

Suppose $\aug$ and $\aug'$ are augmentations in $\mbox{Aug}(\front)$ and there exists a chain homotopy $H$ from $\aug$ to $\aug'$. Suppose $q$ is a degree 0 crossing and $\langle \df q, \prod_{i=1}^m q_{k_i} \rangle$ is $1$, where $\langle \df q, \prod_{i=1}^m q_{k_i} \rangle$ is the coefficient of $\prod_{i=1}^m q_{k_i}$ in $\df q$. Then, by Equation (\ref{eq:chain-homotopy}), 
\begin{eqnarray*}
(\aug - \aug')(q) &=& H \circ \partial (q) \\ \nonumber
&=& H \left (\prod_{i=1}^m q_{k_i} + \hdots \right ) \\ \nonumber
&=& H \left (\prod_{i=1}^m q_{k_i} \right ) + H( \hdots) \\ \nonumber 
&=& \sum_{j=1}^m \left [ \left ( \prod_{i=1}^{j-1} \aug(q_{k_i}) \right ) H(q_{k_j}) \left ( \prod_{i=j+1}^m \aug' (q_{k_i}) \right ) \right ] + H( \hdots). \\ \nonumber
\end{eqnarray*}
At most one term in the sum $$\sum_{j=1}^m \left [ \left ( \prod_{i=1}^{j-1} \aug(q_{k_i}) \right ) H(q_{k_j}) \left ( \prod_{i=j+1}^m \aug' (q_{k_i}) \right ) \right ]$$ may be non-zero, since $\aug$ and $\aug'$ are non-zero only on generators of degree 0 and $H$ is non-zero only on generators of degree $-1$. Note that, for a fixed $j \in \{1, \hdots, m\}$, the term $$ \left ( \prod_{i=1}^{j-1} \aug(q_{k_i}) \right ) H(q_{k_j}) \left ( \prod_{i=j+1}^m \aug' (q_{k_i}) \right ) $$ is non-zero if and only if $H(q_{k_j})=1$ holds and for $1 \leq i \leq j-1$ (resp. $j+1 \leq i \leq m$), the crossing $q_{k_i}$ is augmented by $\aug$ (resp. $\aug'$).

\begin{remark}
The monomials  $\prod_{i=1}^m q_{k_i}$ appearing in $\partial(q)$ correspond to certain mappings of the two-disk $D^2$ into the $xz$-plane that are immersions except for allowable exceptions along $\partial D^2$. Only monomials containing generators of degree $0$ or $-1$ are relevant for our purposes. Therefore, we present only the description of such disks in the following definitions.  Note that this restriction allows us to rule out some additional behaviors of $\partial D^2$ near right cusps that appear in \cite{Ng2003} and lead to monomials that contain generators of degree $1$.
\end{remark}

Let $D^2$ be the disk of radius 1 centered at the origin in $\R^2$. Choose $m$ points from $\df D^2 \setminus \{(1,0)\}$. Label the chosen points $\{b_1, \hdots, b_m\}$ counter-clockwise with $b_1$ the first point counter-clockwise from $(1,0)$. 

\begin{figure}[t]
\labellist
\small\hair 2pt
\pinlabel {(a)} [tl] at 28 78
\pinlabel {(b)} [tl] at 148 78
\pinlabel {(c)} [tl] at 271 44
\pinlabel {(d)} [tl] at 28 -2
\pinlabel {(e)} [tl] at 148 -2
\pinlabel {(f)} [tl] at 271 -2
\pinlabel {i} [tl] at 233 33
\pinlabel {j} [tl] at 233 6
\pinlabel {$\{x_0\} \times [i,j]$} [tl] at 268 20
\endlabellist
\centering
\includegraphics[scale=.9]{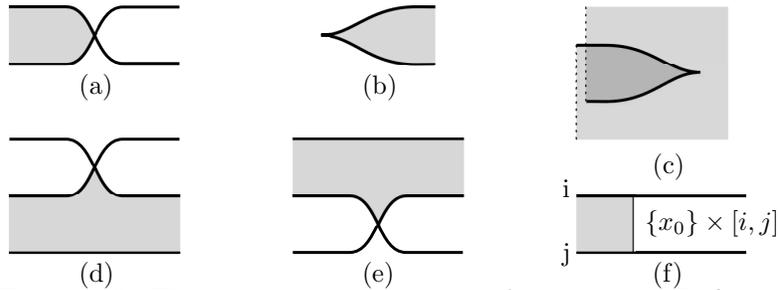}
\caption{The possible singularities of the disk in Definition~\ref{defn:admissible} and the half-disks in Definitions~\ref{defn:aug-half-disk} and \ref{defn:half-disk}. The crossings in (d) and (e) are called convex corners.  Near a boundary point that maps to a right cusp the image of a disk overlaps itself as indicated in (c) by the darkly shaded region.}
\label{f:admissible}
\end{figure}

\begin{definition}
\label{defn:admissible}
In terms of the notation above, a \textbf{$(0,-1)$-admissible disk} is a map from $D^2$ into the $xz$-plane that maps $\df D^2$ to the front diagram $\front$ and is a smooth orientation preserving immersion when restricted to the interior of $D^2$ satisfying the following:

\begin{enumerate}
\item The mapping takes $(1,0)$ to a degree 0 crossing $q$ and the image of $f$ in a neighborhood of $(1,0)$ looks as in Figure~\ref{f:admissible} (a). We say the $(0,-1)$-admissible disk \textbf{originates at $q$};
\item For exactly one $1 \leq j \leq m$, $f(b_j)$ is a degree $-1$ crossing $q_{k_j}$ and the image of $f$ in a neighborhood of $b_j$ looks as in Figure~\ref{f:admissible} (d) or (e).
\item For all $i \neq j$, $f(b_i)$ is a degree 0 crossing $q_{k_i}$ and the image of $f$ in a neighborhood of $b_i$ looks as in Figure~\ref{f:admissible} (d) or (e). 
\item Along $\partial D^2$ the mapping is smooth except at $\{b_1, \hdots, b_m\} \cup \{(1,0)\}$ as described in (1)-(3) and at points in $\df D^2 \setminus (\{b_1, \hdots, b_m\} \cup \{(1,0)\})$ where the image of $f$ looks like either Figure~\ref{f:admissible} (b) or (c).
\end{enumerate}

We say the $(0,-1)$-admissible disk has  \textbf{convex corners} at $q_{k_1}, \hdots, q_{k_m}$. The $(0, -1)$-admissible disk is assigned the monomial $\prod_{i=1}^m q_{k_i}$. We say a $(0,-1)$-admissible disk is an \textbf{$(\aug, \aug', H)$-admissible disk} if, for some $1 \leq j \leq m$, $H(q_{k_j})=1$ holds and for $1 \leq i \leq j-1$ (resp. $j+1 \leq i \leq m$), the crossing $q_{k_i}$ is augmented by $\aug$ (resp. $\aug'$); see Figure \ref{f:eHeDisk}.

Henceforth, we consider admissible disks up to orientation preserving reparametrization of the domain (fixing $\{b_1, \hdots, b_m\} \cup \{(1,0)\}$), and all counts of disks are up to this equivalence relation. 
\end{definition}

\begin{figure}[t]
\labellist
\small\hair 2pt
\pinlabel {$q$} [l] at 170 80
\pinlabel {$\aug$} [bl] at 152 133
\pinlabel {$\aug$} [b] at 109 161
\pinlabel {$\aug$} [b] at 62 161
\pinlabel {$H$} [b] at 15 133
\pinlabel {$\aug'$} [r] at -1 80
\pinlabel {$\aug'$} [tr] at 15 27
\pinlabel {$\aug'$} [t] at 58 -1
\pinlabel {$\aug'$} [t] at 107 -1
\pinlabel {$\aug'$} [tl] at 153 28
\endlabellist
\centering
\includegraphics[scale=.6]{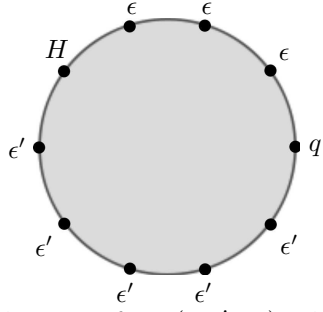}
\caption{The domain of an $(\aug, \aug', H)$-admissible disk with labels indicating marked points mapped to crossings augmented by $\aug$ and $\aug'$ and the marked point mapped to the crossing satisfying $H(q_{k_j})=1$.}
\label{f:eHeDisk}
\end{figure}

The restrictions on the types of non-smooth points of an $(0,-1)$-admissible disk imply that $q$ is the right-most point of the disk. From Section 2 of \cite{Ng2003}, when a single $q_{k_j}$ has degree $1$ while $q$ and all of the remaining $q_{k_i}$ have degree $0$, $\langle \df q, \prod_{i=1}^m q_{k_i} \rangle = 1$ holds if and only if there are an odd number of $(0,-1)$-admissible disks originating at $q$ and with monomial  $\prod_{i=1}^m q_{k_i}$. Proposition~\ref{prop:admissible-disks} follows directly from the discussion above. 

\begin{proposition}
\label{prop:admissible-disks}
Suppose $\front$ is a $\sigma$-generic front diagram of a Legendrian knot and $\aug$ and $\aug'$ are augmentations in $\mbox{Aug}(\front)$. If $q$ is a degree $0$ crossing and $H$ is a chain homotopy from $\aug$ to $\aug'$, then $\aug$ and $\aug'$ differ at $q$ if and only if there are an odd number of $(\aug, \aug', H)$-admissible disks originating at $q$. 
\end{proposition}

\subsection{Morse Complex Sequences}

We briefly sketch the connection between generating families and Morse complex sequences and refer the reader to \cite{Henry2013} for more details. A one-parameter family of smooth functions $f_x: \R^N \to \R$, parameterized by $x \in \R$, is a generating family for a Legendrian knot $\Leg$ with front diagram $\front$ if $$\front = \left \{ (x, z): z = f_x(e) \mbox{ for some } e \in \R^n \mbox{ satisfying } \frac{\df f_x}{\df e}(e) = 0 \right \}.$$ With an appropriately chosen metric, a generic $x \in \R$ determines a Morse chain complex $(C_x, d_x)$ on $\R^N$ and, as $x$ varies, the evolution of the Morse complexes of $f_x$ are well-understood; a cusp of $\front$ corresponds to the creation or elimination of a canceling pair of critical points and a crossing corresponds to two critical points exchanging critical values. As $x$ varies, it is also possible for a fiberwise gradient flowline to momentarily flow between two critical points of the same index. Such an occurrence is called a handleslide and it determines an explicit chain isomorphism between successive Morse complexes. In summary, a generating family and choice of metric determine a one-parameter family of Morse chain complexes and the relationship between successive chain complexes is determined by the crossings and cusps of $\front$ and the handleslides. A Morse complex sequence on $\front$ is a finite sequence of chain complexes $(C_m, d_m)$ and vertical marks on $\front$ that are meant to correspond to the Morse chain complexes and handleslides of a generating family and choice of metric. In addition, varying the choice of metric motivates an equivalence relation on MCSs. 

Fix a Legendrian knot $\Leg$ with $\sigma$-generic front diagram $\front$, rotation number $0$, and Maslov potential $\Maslov$. Theorem~\ref{t:bijection} proves that a certain surjective map in \cite{Henry2011} from equivalence classes of Morse complex sequences to $\mbox{Aug}^{ch}(\front)$ is, in fact, a bijection. The definition of an MCS given in \cite{Henry2011} defines an MCS independent of a fixed front diagram. It is then shown that an MCS determines a front diagram. An alternative definition of an MCS given in \cite{Henry2013} defines an MCS as an object assigned to a fixed front diagram. Both definitions determine the same set of objects on a fixed front diagram. We will use the definition of a Morse complex sequence given in \cite{Henry2013}.  

A \textbf{handleslide} on $\front$ is a vertical line segment disjoint from all crossings and cusps and with endpoints on strands of $\front$ that have the same Maslov potential.

\begin{definition}
\label{defn:MCS}
A \textbf{Morse complex sequence} on a $\sigma$-generic front diagram $\front$ is the triple $\MCS = ( \{(C_m, d_m)\}, \{x_m\}, H)$ satisfying:

\begin{enumerate}
\item $H$ is a set of handleslides on $\front$.
\item The real values $x_1 < x_2 < \hdots < x_M$ are $x$-coordinates distinct from the $x$-coordinates of crossings and cusps of $\front$ and handleslides of $H$. For each $1 \leq m < M$, the set $\{(x, z): x_m \leq x \leq x_{m+1}\}$ contains a single crossing, cusp or handleslide. The set $\{(x, z): -\infty < x \leq x_1\}$ contains the left-most left cusp and the set $\{(x,z): x_M \leq x < \infty\}$ contains the right-most right cusp.  
\item For each $1 \leq m \leq M$, the points of intersection of the vertical line $\{x_m\} \times \R$ and $\front$ are labeled $e_1, e_2, \hdots, e_{s_m}$ from top to bottom. The vector space $C_m$ is the $\Z$-graded $\Z / 2 \Z$ vector space generated by $e_1, e_2, \hdots, e_{s_m}$, where the degree of each generator is the value of the Maslov potential on the corresponding strand of $\front$, $|e_i| = \Maslov(e_i)$. The map $d_m : C_m \to C_m$ is a degree $-1$ differential that is triangular in the sense that $$d_m e_i = \sum_{i<j} c_{ij} e_j, \mbox{ } c_{ij} \in \Z / 2 \Z;$$
\item The coefficients $\langle d_1 e_1, e_2 \rangle$ and $\langle d_M e_1, e_2 \rangle$ are both $1$. Suppose $1 \leq m < M$ and let $T$ be the tangle $\front \cap \{ (x, z) : x_m \leq x \leq x_{m+1}\}$. If $T$ contains a left (resp. right) cusp between strands $k$ and $k+1$, then $\langle d_{m+1} e_k, e_{k+1} \rangle $ is $1$ (resp. $\langle d_m e_k, e_{k+1} \rangle$ is $1$). If $T$ contains a crossing between strands $k$ and $k+1$, then $\langle d_m e_k, e_{k+1}\rangle$ is $0$. 
\item For $1 \leq m < M$, the crossing, cusp, or handleslide mark in the tangle $T = \front \cap \{ (x,z) : x_m \leq x \leq x_{m+1} \}$ determines an algebraic relationship between the chain complexes $(C_m, d_m)$ and $(C_{m+1},d_{m+1})$ as follows:

\begin{enumerate}
\item \textbf{Crossing:} If the crossing is between strands $k$ and $k+1$, then the map $\phi : (C_m, d_m) \to (C_{m+1}, d_{m+1})$ defined by:
\begin{equation*}
\phi(e_i) = \left\{ \begin{array}{rl}
 e_i &\mbox{ if $i \notin \{k, k+1\}$} \\
 e_{k+1} &\mbox{ if $i = k$} \\
 e_{k} &\mbox{ if $i = k+1$} 
       \end{array} \right.
\end{equation*}
is an isomorphism of chain complexes. 

\item \textbf{Right cusp:} If the right cusp is between strands $k$ and $k+1$, then the linear map
\begin{equation*}
	\phi ( e_i) =
	\begin{cases}
  [e_i] &\mbox{ if $i < k $} \\
  [e_{i+2}] &\mbox{ if $i \geq k $}.
  \end{cases}
\end{equation*}
is an isomorphism of chain complexes from $(C_{m+1},d_{m+1})$ to the quotient of $(C_m, d_m)$ by the acyclic subcomplex generated by $\{e_k, d_{m}e_k\}$. 
\item \textbf{Left cusp:} The case of a left cusp is the same as the case of a right cusp, though the roles of $(C_{m}, d_m)$ and $(C_{m+1}, d_{m+1})$ are reversed.
\item \textbf{Handleslide:} If the handleslide mark has endpoints on strands $k$ and $l$ with $k < l$, then the map $h_{k,l}:(C_m, d_m) \to (C_{m+1}, d_{m+1})$ defined by 
\begin{equation*}
	h_{k,l} ( e_i) =
	\begin{cases}
  e_i &\mbox{ if $i \neq k $} \\
  e_k + e_l &\mbox{ if $i=k$}.
  \end{cases}
\end{equation*}
is an isomorphism of chain complexes.

\end{enumerate}

\end{enumerate}

\end{definition}

The set $\mbox{MCS}(\front)$ is the set of all Morse complex sequences on $\front$. 

\begin{remark}
Morse complex sequences may be defined over more general coefficient rings than $\Z/2$, cf. \cite{Henry2014}.  We restrict attention to $\Z/2$ coefficients as this is also done in \cite{Henry2011}.   
\end{remark}

\begin{definition}
\label{defn:simple}
An MCS $\MCS = ( \{(C_m, d_m)\}, \{x_m\}, H)$ in $\mbox{MCS}(\front)$ has \textbf{simple left cusps} if, for each tangle $T = \{ (x,z) : x_m \leq x \leq x_{m+1} \}$ containing a left cusp between strands $k$ and $k+1$, the chain complex $(C_{m+1}, d_{m+1})$ satisfies $\langle d_{m+1} e_k, e_i \rangle = \langle d_{m+1} e_{k+1}, e_i \rangle = 0$ 
for all $k+1 < i$ and $\langle d_{m+1} e_j, e_{k+1} \rangle = \langle d_{m+1} e_j, e_{k} \rangle = 0$ for all $j < k$. 
\end{definition}

The subset $\mbox{MCS}_b(\front) \subset \mbox{MCS}(\front)$ denotes the set of MCSs with simple left cusps. We use the letter $b$ to be consistent with the notation of \cite{Henry2011}, where a left cusp is also called a ``birth''. This language is meant to draw a connection to the creation of a canceling pair of critical points, often called a birth, in a one-parameter family of Morse functions on a manifold. 

Given an MCS $\MCS = ( \{(C_m, d_m)\}, \{x_m\}, H)$ with simple left cusps, the chain complexes $\{(C_m, d_m)\}$ are uniquely determined by the crossings and cusps of $\front$, the handleslides $H$, and requirements (5) (a)-(d) of Definition~\ref{defn:MCS}. Consequently, $\MCS$ may be represented visually by placing the handleslide marks $H$ on the front diagram $\front$; see Figure~\ref{f:MCS-example}.

\begin{figure}[t]
\centering
\includegraphics[scale=.9]{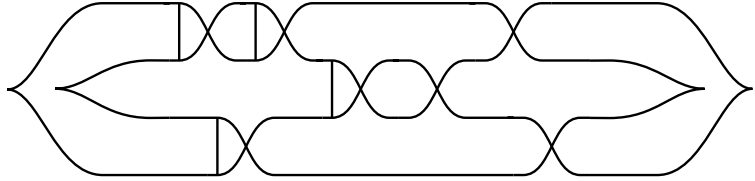}
\caption{An MCS with simple left cusps. This MCS is also in A-form.}
\label{f:MCS-example}
\end{figure}

\begin{figure}[t]
\labellist
\small\hair 2pt
\pinlabel {(1)} [tl] at 69 403
\pinlabel {(2)} [tl] at 261 403
\pinlabel {(3)} [tl] at 69 324
\pinlabel {(4)} [tl] at 261 324
\pinlabel {(5)} [tl] at 69 251
\pinlabel {(6)} [tl] at 261 251
\pinlabel {(7)} [tl] at 69 172
\pinlabel {(8)} [tl] at 261 172
\pinlabel {(9)} [tl] at 70 101
\pinlabel {(10)} [tl] at 258 101
\pinlabel {(11)} [tl] at 65 20
\pinlabel {(12)} [tl] at 258 20
\endlabellist
\centering
\includegraphics[scale=.9]{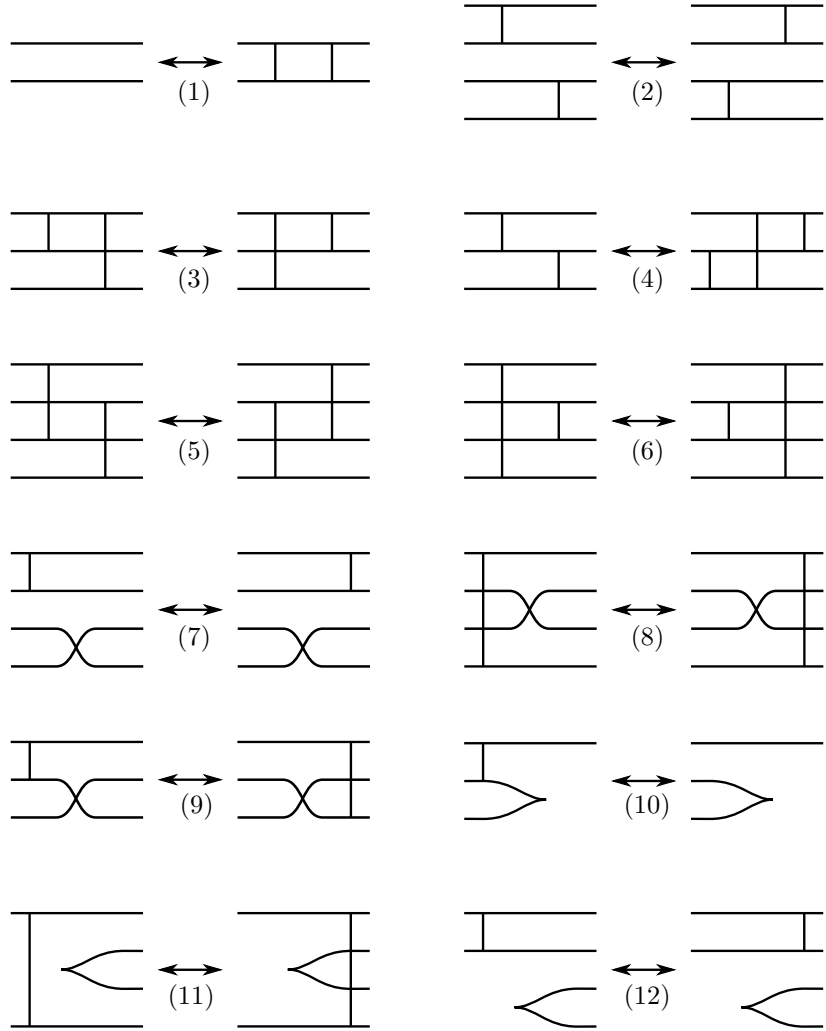}
\caption{Handleslide modifications, called MCS moves, that result in an equivalent MCS.}
\label{f:moves}
\end{figure}

\begin{figure}[t]
\labellist
\small\hair 2pt
\pinlabel {$i$} [tr] at -1 89
\pinlabel {$k$} [tr] at -1 57
\pinlabel {$l$} [tr] at -1 40
\pinlabel {$j$} [tr] at -1 9
\endlabellist
\centering
\includegraphics[scale=1.2]{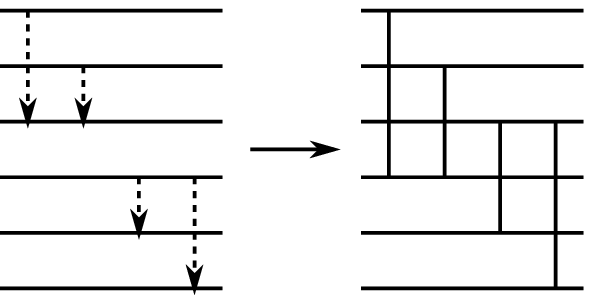}
\caption{MCS move (13). On the left, a dotted arrow from strand $\alpha$ to strand $\beta$ indicates that $\langle d_m e_{\alpha}, e_{\beta} \rangle$ is $1$.}
\label{f:explosion}
\end{figure}

In \cite{Henry2011}  an equivalence relation on the set $\mbox{MCS}(\front)$ is defined that is motivated by a corresponding equivalence for generating families.  (See also \cite{Henry2013}.)  Here we denote the set of equivalence classes of this relation by  $\widehat{\mbox{MCS}}(\front) = \mbox{MCS}(\front) / \simeq$.   We recall a version of this equivalence relation that applies to the more restricted set of MCSs with simple left cusps, $\mbox{MCS}_b(\front)$.  We denote equivalence classes with respect to this relation by $\widehat{\mbox{MCS}}_b(\front)$.  By Proposition 3.17 of \cite{Henry2011}, the map from $\widehat{\mbox{MCS}}_b(\front)$ to $\widehat{\mbox{MCS}}(\front)$ induced by the inclusion  $\mbox{MCS}_b(\front) \subset \mbox{MCS}(\front)$ is a bijection.  Therefore, to prove Theorem~\ref{t:bijection}, we need only consider MCSs in $\mbox{MCS}_b(\front)$ and MCS classes in $\widehat{\mbox{MCS}}_b(\front)$. 

The equivalence relation on $\mbox{MCS}_b(\front)$ is generated by the \textbf{MCS moves} pictured in Figures~\ref{f:moves} and \ref{f:explosion}.  The numbering indicated will be used throughout this article. 
Additional moves result from reflecting each of the two figures in (3),(7),(9),(10), and (12) of Figure~\ref{f:moves} about a horizontal axis and reflecting each of the two figures in (4), (9), (11), and (12) of Figure~\ref{f:moves} about a vertical axis. The handleslide modification that results from reflecting Figure~\ref{f:moves} (10) about a vertical axis is \emph{not} an MCS move for MCSs with simple left cusps.  (The absence of this reflected move is the only difference between the definitions of the equivalence relations on $\mbox{MCS}_b(\front)$ and  $\mbox{MCS}(\front)$ discussed in the previous paragraph.) MCS move (13) requires explanation. Suppose $\MCS = ( \{(C_m, d_m)\}, \{x_m\}, H)$ is an MCS on $\front$ and suppose there exists $x_m$ and $1 \leq k < l \leq s_m$ so that $\Maslov(e_k) = \Maslov(e_{l})-1$, then MCS move (13) introduces the collection of handleslides $K$ defined as follows.  The handleslides in $K$ are of two types. First, if $i < k$ and $\langle d_m e_i, e_k \rangle=1$ holds, then $K$ contains a handleslide with endpoints on $i$ and $l$. Second, if $l<j$ and $\langle d_m e_{l}, e_j \rangle=1$ holds, then $K$ contains a handleslide with endpoints on $k$ and $j$.

By Proposition 3.8 of \cite{Henry2011}, modifying the handleslide set of an MCS in $\mbox{MCS}_b(\front)$ as in one of the cases in Figures~\ref{f:moves} and \ref{f:explosion} results in another MCS in $\mbox{MCS}_b(\front)$. Therefore, the notion of equivalence in the following definition is well-defined. In addition, if an MCS move is applied to an MCS, then only those chain complexes near the location of the MCS move are affected. In other words, the MCS moves are local in the sense that they change both the handleslides and chain complexes of an MCS only in a local neighborhood.

\begin{definition}
\label{defn:equiv}
Two MCSs $\MCS$ and $\MCS'$ in $\mbox{MCS}_b(\front)$ are \textbf{equivalent}, written $\MCS \simeq \MCS'$, if there exists a sequence $\MCS_1, \MCS_2, \hdots, \MCS_s$ in $\mbox{MCS}_b(\front)$ so that $\MCS = \MCS_1$, $\MCS'=\MCS_s$, and, for all $1 \leq i < s$, the set of handleslide marks of $\MCS_i$ and $\MCS_{i+1}$ differ by exactly one MCS move. The set $\widehat{\mbox{MCS}}_b(\front)$ is the set $\mbox{MCS}_b(\front) / \simeq$.
\end{definition}

MCSs of the following type have a standard form that makes their relationship with augmentations particularly simple to describe.

\begin{definition}
\label{defn:A-form}
An MCS $\MCS$ in $\mbox{MCS}_b(\front)$ is in \textbf{A-form} if there exists a set $R$ of degree $0$ crossings so that just to the left of each $q$ in $R$ there is a handleslide with endpoints on the strands crossing at $q$ and $\MCS$ has no other handleslides. A crossing $q$ in $R$ is said to be \textbf{marked}. 

Figure~\ref{f:MCS-example} shows an MCS in A-form where $R$ is the four left-most crossings. The subset $\mbox{MCS}_A(\front) \subset \mbox{MCS}_b(\front)$ consists of all A-form MCSs on $\front$. 
\end{definition}

\section{The Main Result}
\label{s:results}

Suppose $\Leg$ is a Legendrian knot with $\sigma$-generic front diagram $\front$, rotation number 0, and Maslov potential $\Maslov$. Before proving Theorem~\ref{t:bijection}, we require two definitions and a technical lemma. 

Let $D^2$ be the disk of radius 1 centered at the origin in $\R^2$. Choose $m+2$ points on $\df D^2$. Label the chosen points $\{b_0, \hdots, b_{m+1}\}$ counter-clockwise. Let $\gamma$ be the arc of $\df D^2$ with endpoints $b_{m+1}$ and $b_0$ and so that $b_1, \hdots, b_m$ are not in $\gamma$. Given $x_0 \in \R$ that is not the $x$-coordinate of any crossing or cusp of $\front$, we let $\{x_0\} \times [i,j]$ denote the vertical line segment with $x$-coordinate $x_0$ and endpoints on strands $i$ and $j$ of $\front$, where the strands of $\front$ above $x=x_0$ are numbered $1, 2, \hdots$ from top to bottom and $i < j$.

\begin{definition}
\label{defn:half-disk}
Let $\aug$ and $\aug'$ be homotopic augmentations in $\mbox{Aug}(\front)$ and let $H$ be a chain homotopy from $\aug$ to $\aug'$. An \textbf{$(\aug,\aug',H)$-half-disk} is a mapping of the two-disk $D^2$ into the $xz$-plane as in Definition \ref{defn:admissible} except for the following variations along the boundary:
\begin{enumerate}
\item The arc $\gamma$ maps to a vertical line $\{x_0\} \times [i,j]$ where $\Maslov(i)=\Maslov(j)$; see Figure~\ref{f:admissible} (f). We say the $(\aug,\aug',H)$-half-disk \textbf{originates at $\{x_0\} \times [i,j]$};
\item For exactly one $1 \leq j \leq m$, $f(b_j)$ is a degree $-1$ crossing $q_{k_j}$, $H(q_{k_j})=1$ holds, and $f$ has a convex corner at $f(b_j)$; see Figure~\ref{f:admissible} (d) or (e);
\item If $1 \leq i < j$ (resp. $j < i \leq m$), $f(b_i)$ is a degree $0$ crossing augmented by $\aug$ (resp. $\aug'$) and $f$ has a convex corner at $f(b_i)$;
\item The restriction of $f$ to $\partial D^2$ is  smooth except at $\{b_0, \hdots, b_{m+1}\}$ as described in (1) and (2) and at points in $\df D^2 \setminus (\{b_0, \hdots, b_{m+1}\})$ where the image of $f$ looks like Figure~\ref{f:admissible} (b) or (c).
\end{enumerate}
The set $\mathcal{H}(x_0, [i, j])$ consists of all $(\aug,\aug',H)$-half-disks originating at $\{x_0\}\times[i,j]$ up to reparametrization,  and $\# \mathcal{H}(x_0, [i, j])$ is the mod 2 count of elements in $\mathcal{H}(x_0, [i, j])$.
\end{definition}

\begin{definition}
\label{defn:aug-half-disk}
Let $\aug$ be an augmentation in $\mbox{Aug}(\front)$. An \textbf{$\aug$-half-disk} is a mapping $f$ of the two-disk $D^2$ into the $xz$-plane as in Definition \ref{defn:half-disk} except that conditions (2) and (3) are replaced with the requirement that all convex corners are at crossings that are augmented by $\aug$.

The set $\mathcal{G}^{\aug}(x_0, [i, j])$ consists of all $\aug$-half-disks originating at $\{x_0\}\times[i,j]$ up to reparametrization, and $\# \mathcal{G}^{\aug}(x_0, [i, j])$ is the mod 2 count of elements in $\mathcal{G}^{\aug}(x_0, [i, j])$.
\end{definition}

As in Definition~\ref{defn:admissible}, the points in the vertical line $\{x_0\} \times [i,j]$ are the right-most points of either an $(\aug, \aug', H)$-half-disk or an $\aug$-half-disk.

By Corollary 6.21 of \cite{Henry2011}, the map $\Phi: \mbox{MCS}_A(\front) \to \mbox{Aug}(\front)$ defined as follows is a bijection. Given $\MCS \in \mbox{MCS}_A(\front)$ and a generator $q$ of $\alg(\front)$, $\Phi(\MCS)(q) = 1$ holds if and only if $q$ is a marked crossing of $\MCS$. We let $\aug_{\MCS}$ be the augmentation $\Phi(\MCS)$. 

Lemma~\ref{l:aug-half-disks} below generalizes Lemma 7.10 in \cite{Henry2013} and Lemma 5.4 in \cite{Henry2014} by removing the assumption that the front diagram $D$ is nearly plat. 
Note that ``gradient paths'' from Lemma 7.10  in \cite{Henry2013} correspond to $\aug_{\MCS}$-half-disks in our terminology, and that Lemma 5.4 of \cite{Henry2014} allows more general coefficients.

\begin{lemma}
\label{l:aug-half-disks}
Suppose $\front$ is a $\sigma$-generic front diagram and $\MCS = ( \{C_m, d_m\}, \{x_m\}, H)$ is in $\mbox{MCS}_A(\front)$. Suppose $p \in \{1, \hdots, M\}$ and $x_p$ is to the immediate right of a crossing or cusp. Then, for all $i < j$, 
\begin{equation} \label{eq:mathcalG}
\langle d_p e_i, e_j \rangle= \#\mathcal{G}^{\aug_{\MCS}}(x_p, [i, j]) \mbox{ holds.}
\end{equation}
\end{lemma}

\begin{proof}
We induct on $p$.  The base case, $p=1$, follows since there is a unique disk in $\mathcal{G}^{\aug_{\MCS}}(x_1, [1, 2])$, as in Figure \ref{f:admissible} (b), while $\langle d_1 e_1, e_2 \rangle=1$ holds according to (4) of Definition \ref{defn:MCS}.  

Assume now that $x_p$ sits to the immediate right of a crossing or cusp and that the result is known for smaller values of $p$.  We complete the inductive step by considering cases.

{\bf Left cusp.}  Suppose $x_p$ is to the right of a left cusp with the two strands that meet at the cusp labeled $k$ and $k+1$ at $x_p$.  Define $\tau: \{1,\ldots, s_{p-1}\} \rightarrow \{1, \ldots, s_p\}$ by $\tau(i) =\left\{ \begin{array}{cr} i,  & \mbox{if $i<k$} \\ i+2 & \mbox{if $i\geq k$.}  \end{array} \right.$  (Note that $s_{p-1} = s_p -2$.)  For any $1\leq i' < j' \leq s_{p-1}$ there is a bijection between $\mathcal{G}^{\aug_{\MCS}}(x_{p-1}, [i', j'])$ and $\mathcal{G}^{\aug_{\MCS}}(x_p, [\tau(i'), \tau(j')])$; see, for example, Figure~\ref{f:leftcuspcase}. Moreover, (5)(c) of Definition \ref{defn:MCS} together with the requirement that $\MCS$ has simple left cusps give 
\[
\langle d_{p-1} e_{i'}, e_{j'} \rangle = \langle d_{p} e_{\tau(i')}, e_{\tau(j')} \rangle,
\]
so (\ref{eq:mathcalG}) follows  when $i= \tau(i')$ and $j=\tau(j')$.  

\begin{figure}[t]
\labellist
\small\hair 2pt
\pinlabel {(a)} [tl] at 76 25
\pinlabel {(b)} [tl] at 280 25
\endlabellist
\centering
\includegraphics[scale=.9]{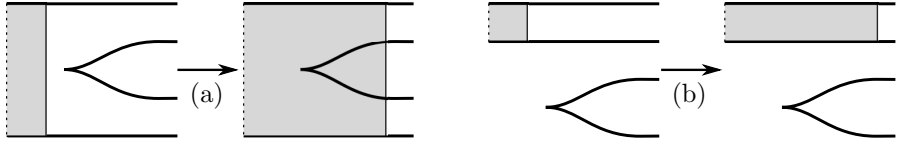}
\caption{Possible extensions of an $\aug$-half-disk or $(\aug, \aug', H)$-half-disk past a left cusp.}
\label{f:leftcuspcase}
\end{figure}

It remains to consider those cases where $\{i,j\} \cap \{k,k+1\} \neq \emptyset$.  Suppose that precisely one of $i$ or $j$ belongs to $\{k,k+1\}$. As $\MCS$ has simple left cusps, we have $\langle d_{p} e_{i}, e_{j} \rangle=0$.  In addition, the restriction on the behavior of an $\aug$-half disk near a left cusp from Figure \ref{f:admissible} (b) gives that $\mathcal{G}^{\aug_{\MCS}}(x_p, [i, j])= \emptyset$, so (\ref{eq:mathcalG}) holds.  Finally, when $i=k$ and $j=k+1$, there is a unique $\aug$-half disk in $\mathcal{G}^{\aug_{\MCS}}(x_p, [k, k+1])$.  (This disk has no convex corners, so (2) of Definition \ref{defn:aug-half-disk} is vacuously satisfied.)  Therefore, (\ref{eq:mathcalG}) follows in view of (4) from Definition \ref{defn:MCS}.

{\bf Crossing.}  When $x_p$ sits immediately to the right of a crossing, the inductive step is achieved precisely as in Lemma 7.10 of \cite{Henry2013} or  Lemma 5.4 of \cite{Henry2014} (the signs in the latter reference may be ignored). The arguments in these references apply regardless of whether or not the crossing is marked. 

{\bf Right cusp.} Suppose a right cusp sits between $x_p$ and $x_{p-1}$ with the strands that meet at the cusp labeled $k$ and $k+1$ at $x_{p-1}$. Let $a_{i,j}$ be $\langle d_{p-1} e_i, e_j\rangle$. In the quotient of $(C_{p-1},d_{p-1})$ by the subcomplex spanned by $e_k$ and $d_{p-1} e_k$, we have
\[
0 = [d_{p-1}e_{k}] = [e_{k+1}] + \sum_{k+1 < j} a_{k,j} [e_j],
\]
 so
\[
d_{p-1} [e_i] = \sum_{i < j} a_{i,j} [e_j] = \sum_{i< j< k} a_{i,j} [e_j] + \sum_{k+1 < j} (a_{i,j} + a_{i,k+1}\cdot a_{k,j})  [e_j].
\]
Using Definition \ref{defn:MCS}~(5)~(b), this gives the computation of the differential in $(C_p,d_p)$ as
\begin{equation}  \label{eq:RCuspCompd}
\langle d_p e_{i}, e_{j}\rangle = \langle d_{p-1}e_{\pi(i)}, e_{\pi(j)} \rangle + \langle d_{p-1} e_{\pi(i)}, e_{k+1} \rangle\cdot \langle d_{p-1} e_{k}, e_{\pi(j)}\rangle,
\end{equation}
where $\pi: \{1,\ldots, s_{p}\} \rightarrow \{1, \ldots, s_{p-1}\}$ is defined by $\pi(i) =\left\{ \begin{array}{cr} i,  & \mbox{if $i<k$} \\ i+2 & \mbox{if $i\geq k$.} \end{array} \right.$ 
We note that the second term on the right can be non-zero only if $i<k \leq j$; see Figure \ref{f:epsilonLemma}.

\begin{figure}[t]
\labellist
\small\hair 2pt
\pinlabel {$x_{p-1}$} [t] at 32 -2
\pinlabel {$x_p$} [t] at 104 -2
\pinlabel {$a_{i,j}+a_{i,k}\cdot a_{i,k+1}$} [l] at 106 52
\pinlabel {$a_{i,k+1}$} [l] at 34 64
\pinlabel {$a_{i,j}$} [r] at 14 64
\pinlabel {$a_{k,j}$} [l] at 50 44
\pinlabel {$i$} [l] at 278 72
\pinlabel {$j$} [l] at 278 16
\endlabellist
\centering
\includegraphics[scale=1.2]{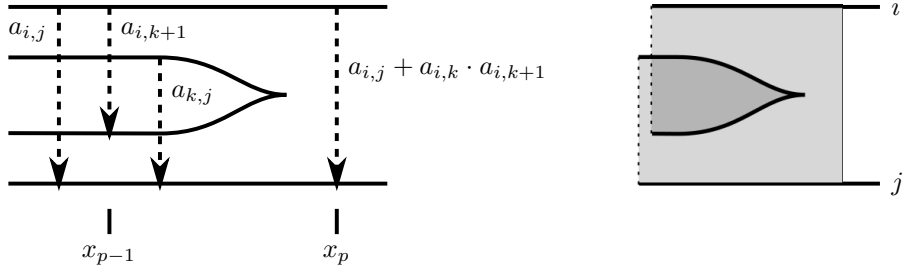}
\caption{(left) The relation between differentials at a right cusp.  A dotted arrow at $x_l$ pointing from strand $i$ to strand $j$ indicates the matrix coefficient $\langle d_l e_i, e_j \rangle$.  (right)  The appearance of disks  in $\mathcal{G}^{\aug_{\MCS}}(x_p, [i, j])$ with a boundary point at the right cusp between $x_{p-1}$ and $x_p$.}
\label{f:epsilonLemma}
\end{figure}

\begin{figure}[t]
\labellist
\small\hair 2pt
\pinlabel {(a)} [tl] at 66 100
\pinlabel {(b)} [tl] at 254 100
\pinlabel {(c)} [tl] at 65 24
\endlabellist
\centering
\includegraphics[scale=.9]{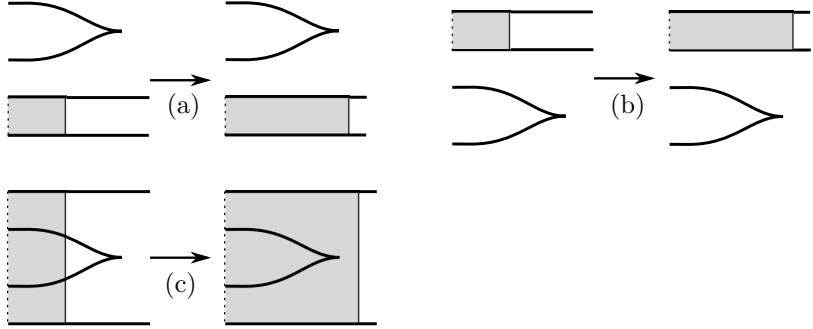}
\caption{Possible extensions of an $\aug$-half-disk or $(\aug, \aug', H)$-half-disk past a right cusp.}
\label{f:cuspcase}
\end{figure}

To complete the proof, we combine Equation (\ref{eq:RCuspCompd}) with the observation that $\aug$-half disks satisfy a bijection
\begin{align*}
\mathcal{G}^{\aug_{\MCS}}(x_p, [i, j]) &\cong \mathcal{G}^{\aug_{\MCS}}(x_{p-1}, [\pi(i), \pi(j)]) \\ \nonumber
&\cup \left(\mathcal{G}^{\aug_{\MCS}}(x_{p-1}, [\pi(i),k+1])\times \mathcal{G}^{\aug_{\MCS}}(x_{p-1}, [k,\pi(j)]) \right)
\end{align*}
explained as follows:  Those disks in $\mathcal{G}^{\aug_{\MCS}}(x_p, [i, j])$ whose boundaries do not intersect the cusp point are in bijection with $\mathcal{G}^{\aug_{\MCS}}(x_{p-1}, [\pi(i), \pi(j)])$; see Figure~\ref{f:cuspcase}.  Disks in $\mathcal{G}^{\aug_{\MCS}}(x_p, [i, j])$ whose boundaries do intersect the cusp point appear between $x_{p-1}$ and $x_p$ as pictured in Figure \ref{f:epsilonLemma}.  Removing the portion of the disk between  $x_{p-1}$ and $x_p$ leaves a pair of initially overlapping disks from $\mathcal{G}^{\aug_{\MCS}}(x_{p-1}, [\pi(i),k+1])\times \mathcal{G}^{\aug_{\MCS}}(x_{p-1}, [k,\pi(j)])$, and this correspondence is bijective.
\end{proof}


We are now in a position to prove Theorem~\ref{t:bijection}. Recall that an augmentation $\aug$ has an associated A-form MCS $\MCS$ where a degree $0$ crossing $q$ is marked by $\MCS$ if and only if $\aug(q)$ is $1$. The proof of Theorem~\ref{t:bijection} reduces to showing that if augmentations $\aug$ and $\aug'$ are homotopic, then their associated A-form MCSs $\MCS$ and $\MCS'$ are equivalent. This is accomplished in Lemma~\ref{l:equiv} where an algorithm is given to translate a chain homotopy $H$ from $\aug$ to $\aug'$ into a sequence of MCS moves from $\MCS$ to $\MCS'$. In particular, for each degree $-1$ crossing $p$ sent to $1$ by $H$, we employ MCS move (13) just to the left of $p$ to introduce new handleslides. We prove that these handleslides give the mod 2 count of certain $(\aug, \aug', H)$-half disks. Moving these handleslides to the right in the front diagram $\front$, we find that a degree $0$ crossing $q$ is changed from marked to unmarked or from unmarked to marked if and only if there exists an odd number of $(\aug, \aug', H)$-half disks originating at $q$. Therefore, by Proposition~\ref{prop:admissible-disks} and the definition of $\MCS$ and $\MCS'$, $q$ is changed from marked to unmarked or from unmarked to marked if and only if $\MCS$ and $\MCS'$ differ at $q$. We may therefore conclude that $\MCS$ and $\MCS'$ are equivalent.  

\begin{proof}[Proof of Theorem~\ref{t:bijection}]
By Proposition 3.17 of \cite{Henry2011}, the natural inclusion of $\mbox{MCS}_b(\front)$ into $\mbox{MCS}(\front)$ induces a bijection from $\widehat{\mbox{MCS}}_b(\front)$ to $\widehat{\mbox{MCS}}(\front)$. Therefore, it suffices to construct a bijection from $\widehat{\mbox{MCS}}_b(\front)$ to $\mbox{Aug}^{ch}(\front)$. In Section 6 of \cite{Henry2011}, a surjective map $\widehat{\Psi}$ is constructed from $\widehat{\mbox{MCS}}_b(\front)$ to $\mbox{Aug}^{ch}(\front)$. We will prove this map is injective. By Theorem 1.6 of \cite{Henry2011}, every MCS is equivalent to an A-form MCS. Therefore, every MCS equivalence class contains an A-form representative. We give the definition of $\widehat{\Psi}$ in terms of A-form representatives and, in so doing, avoid most of the technical details of \cite{Henry2011}. By Corollary 6.21 of \cite{Henry2011}, given an MCS class $[\MCS]$ with A-form representative $\MCS$, $\widehat{\Psi}([\MCS])$ is the augmentation homotopy class $[\aug_{\MCS}]$ where a degree 0 crossing $q$ is augmented by $\aug_{\MCS}$ if and only if $q$ is marked by $\MCS$. Lemma~\ref{l:equiv} shows that if $\aug_{\MCS_1}$ is homotopic to $\aug_{\MCS_2}$, then $\MCS_1$ is equivalent to $\MCS_2$. It follows that $\widehat{\Psi}$ is injective. 
\end{proof}

\begin{lemma}
\label{l:equiv}
Suppose $\front$ is a $\sigma$-generic front diagram and $\MCS$ and $\MCS'$ are in $\mbox{MCS}_A(\front)$ with corresponding augmentations $\aug_{\MCS}$ and $\aug_{\MCS'}$, respectively. If $\aug_{\MCS}$ and $\aug_{\MCS'}$ are homotopic, then $\MCS$ and $\MCS'$ are equivalent as MCSs.
\end{lemma}

\begin{proof}

Suppose $\MCS$ and $\MCS'$ are A-form MCSs and $\aug_{\MCS}$ is homotopic to $\aug_{\MCS'}$. We simplify notation by letting $\aug$ be $\aug_{\MCS}$ and $\aug'$ be $\aug_{\MCS'}$. Since $\aug$ and $\aug'$ are homotopic, there exists a chain homotopy $H : \alg(\front) \to \Z / 2 \Z$. Label the degree $-1$ crossings sent to 1 by $H$, from left to right, $p_1, \hdots, p_m$.

\begin{figure}[t]
\centering
\includegraphics[scale=.9]{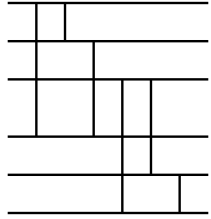}
\caption{An ordered collection of handleslides.}
\label{f:ordered}
\end{figure}

To prove the lemma, we will construct a sequence of MCSs $\MCS_0, \hdots, \MCS_s$ so that $\MCS_0$ is $\MCS$ and $\MCS_s$ is $\MCS'$, and, for all $0 \leq r < s$, $\MCS_r \simeq \MCS_{r+1}$ holds. The construction of the $\MCS_r$ is inductive, and each of the $\MCS_r$ will contain a, possibly empty, collection of handleslides $V_r$ that are grouped together immediately to the right of a particular crossing or cusp.  

For our purposes it will be convenient to require that the handleslides in each of the $V_r$ are ordered in the following sense:   We say a collection of handleslides is \textbf{ordered} if, given two handleslides $h$ and $h'$ in the collection with endpoints on strands $i<j$ and $i'<j'$ respectively, $h$ is right of $h'$ if and only if $i>i'$ holds, or $i=i'$ and $j < j'$ hold; see Figure~\ref{f:ordered}. We let $v_r^{i,j}$ be $1$ if there exists a handleslide in $V_r$ with endpoints on strands $i$ and $j$, where $i < j$. Otherwise, $v_r^{i,j}$ is defined to be $0$. In a slight abuse of notation, we also let $v_r^{i,j}$ refer to the handleslide in $V_r$ with endpoints on $i$ and $j$, if such a handleslide exists. 
 
We will verify that Property~\ref{prop:count} below holds for all $0 \leq r \leq s$ as we inductively construct MCSs $\MCS_r$ with ordered handleslide collections $V_r$.  

\begin{property}
\label{prop:count}
\begin{enumerate}
	\item[(a)] The MCS $\MCS_r$ agrees with $\MCS'$ to the left of $V_r$ and $\MCS$ to the right of $V_r$.
	\item[(b)] For all $i < j$, $$v_r^{i,j} = \# \mathcal{H}(x_r, [i, j])$$ holds, where $x_r$ is the $x$-coordinate of the left-most handleslide in $V_r$. 
\end{enumerate}
\end{property}
  
Each time $r$ increases, the collection of handleslides $V_r$ is pushed to the right past one cusp or crossing. We continue this inductive process until we arrive at an MCS $\MCS_s$ with  $V_s$ located just to the left of the right-most right cusp of $\front$.  Since the two strands of this cusp do not have the same Maslov potential, $V_s$ must be empty. Then, Property~\ref{prop:count} (a) shows that $\MCS_s$ is $\MCS'$. As $\MCS = \MCS_0$, and  for $0 \leq i < s$, $\MCS_i \simeq \MCS_{i+1}$ it will then follow that $\MCS \simeq \MCS'$ holds, as desired.

In the remainder of the proof we construct the sequence of MCSs $\MCS_0, \hdots, \MCS_s$.
Since a crossing $q$ is augmented by $\aug$ (resp. $\aug'$) if and only if $q$ is marked by $\MCS$ (resp. $\MCS'$), Proposition~\ref{prop:admissible-disks} implies $\MCS$ and $\MCS'$ differ at $q$ if and only if there exists an odd number of $(\aug, \aug', H)$-admissible disks originating at $q$. Since $q$ is the right-most point of an $(\aug, \aug', H)$-admissible disks originating at $q$, it follows that there are no admissible disks originating to the left of $p_1$ (which is the first crossing sent to $1$ by $H$).  Therefore, $\MCS$ and $\MCS'$ are identical to the left of $p_1$.  We can then set $\MCS_0= \MCS$ and define $V_0$ to be empty, but located just to the left of $p_1$. It follows that Property \ref{prop:count} holds for $\MCS_0$. 

Given $\MCS_r$ and $V_r$, we will construct $\MCS_{r+1}$ and $V_{r+1}$ by applying MCS moves to $\MCS_r$. We will prove that if Property~\ref{prop:count} holds for $\MCS_r$, then it holds for $\MCS_{r+1}$ as well. We consider five cases depending on the type of crossing or cusp just to the right of $V_r$. Let $q$ be the crossing or cusp point to the immediate right of $V_r$. Let $x_r$ (resp. $x_{r+1}$) be an $x$-coordinate to the immediate left (resp. right) of $q$. In each of the five cases, we first analyze the $(\aug, \aug', H)$-half-disks in $\mathcal{H}(x_{r+1}, [i, j])$ before describing the sequence of MCS moves used to construct $\MCS_{r+1}$ from $\MCS_r$ and proving Property~\ref{prop:count} holds for $\MCS_{r+1}$ and $V_{r+1}$.

In the first three cases considered, $q$ is a crossing between strands $k$ and $k+1$ where the strands of $\front$ have been numbered $1, \hdots, s_r$, from top to bottom, just to the left of $q$. Let $\rho : \{1, \hdots, s_r\} \to \{1, \hdots, s_r\}$ be the permutation that transposes $k$ and $k+1$.

\begin{figure}[t]
\labellist
\small\hair 2pt
\pinlabel {(a)} [tl] at 76 226
\pinlabel {(b)} [tl] at 297 226
\pinlabel {(c)} [tl] at 76 130
\pinlabel {(d)} [tl] at 297 130
\pinlabel {(e)} [tl] at 76 33
\pinlabel {(f)} [tl] at 297 33
\endlabellist
\centering
\includegraphics[scale=.9]{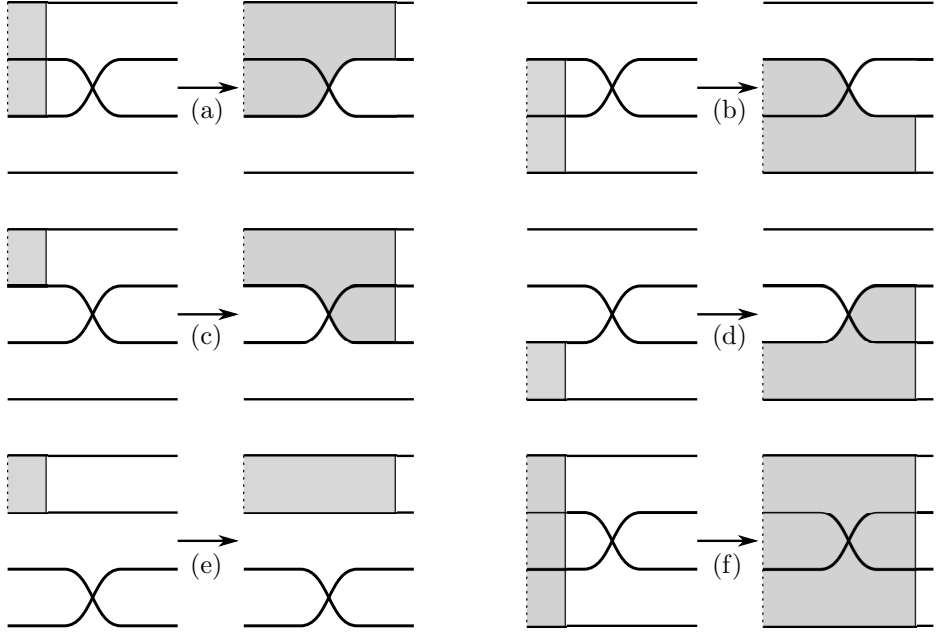}
\caption{Possible extensions of an $(\aug, \aug', H)$-half-disk past a crossing.}
\label{f:simplecrossing}
\end{figure}

\textbf{Crossing $q$ such that $|q|\neq 0$ and $H(q) \neq 1$:} 
Since $|q|$ is non-zero, $\mathcal{H}(x_{r}, [k, k+1])$ and $\mathcal{H}(x_{r+1}, [k, k+1])$ are both empty. Given $1 \leq i < j \leq s_r$ such that $(i,j) \neq (k,k+1)$, $(\aug, \aug', H)$-half-disks in $\mathcal{H}(x_{r+1}, [i, j])$ cannot have a convex corner at $q$, since $|q| \neq 0$ and $H(q) \neq 1$ hold. In fact, $\# \mathcal{H}(x_{r+1}, [i, j]) = \# \mathcal{H}(x_r, [\rho(i), \rho(j)])$ holds, since there is a natural bijection between $\mathcal{H}(x_{r+1}, [i, j])$ and $\mathcal{H}(x_r, [\rho(i), \rho(j)])$; see, for example, Figure~\ref{f:simplecrossing}. 

We now define the sequence of MCS moves that create $\MCS_{r+1}$ from $\MCS_r$ and prove Property~\ref{prop:count} holds for $\MCS_{r+1}$. Move all handleslides of $V_r$ to the right of $q$ using MCS moves (7) - (9). Note that, since $|q|$ is non-zero, $v_r^{k,k+1}$ is $0$ and, therefore, all handleslides of $V_r$ can be moved past $q$ and no new handleslides are created by doing so. The resulting collection can be ordered, using MCS moves, without creating new handleslides.  The reordering requires rearranging handleslides with one endpoint on either strand $k$ or $k+1$. Since $|q|$ is non-zero, there is no handleslide between $k$ and $k+1$, and, therefore, the rearrangement can be done without using MCS Move (4). The resulting ordered collection is $V_{r+1}$ and the MCS is $\MCS_{r+1}$, and $v_{r+1}^{i,j} = v_r^{\rho(i), \rho(j)}$ holds for all $1 \leq i < j \leq s_r$.  By Property~\ref{prop:count} (b), $v_r^{\rho(i), \rho(j)} = \# \mathcal{H}(x_r, [\rho(i), \rho(j)])$ holds and, as shown above, $\# \mathcal{H}(x_{r+1}, [i, j]) = \# \mathcal{H}(x_r, [\rho(i), \rho(j)])$ holds. Therefore, Property~\ref{prop:count} (b) holds for $\MCS_{r+1}$. Property~\ref{prop:count} (a) holds for $\MCS_r$ and, since $|q|$ is non-zero, $q$ is not marked by either $\MCS_{r+1}$ or $\MCS'$. Therefore, Property~\ref{prop:count} (a) holds for $\MCS_{r+1}$.

\textbf{Crossing $q$ such that $|q| = 0$:} Let $v_q$ be 1 if $q$ is marked by $\MCS$ and $0$ otherwise. Since Property~\ref{prop:count} (a) holds for $\MCS_r$, if $q$ is marked by $\MCS$, then $q$ is marked by $\MCS_r$ as well. We slightly abuse notation and also let $v_q$ be the handleslide at $q$ in $\MCS_r$ in the case such exists. 

Suppose $i \neq k+1$ and $j \neq k$. Half-disks in $\mathcal{H}(x_{r+1}, [i, j])$ cannot have a convex corner at $q$ and, therefore, there is a bijection from $\mathcal{H}(x_{r}, [i,  j])$ to $\mathcal{H}(x_{r+1}, [\rho(i), \rho(j)])$; see, for example, Figure~\ref{f:simplecrossing} (c) - (f). Since Property~\ref{prop:count} (b) holds for $\MCS_r$, $\# \mathcal{H}(x_{r+1}, [i, j]) = v_r^{\rho(i), \rho(j)}$ holds. 

\begin{figure}[t]
\labellist
\small\hair 2pt
\pinlabel {(a)} [tl] at 74 33
\pinlabel {(b)} [tl] at 293 33
\pinlabel {$v_q$} [bl] at 235 52
\pinlabel {$v_r^{k,k+1}+v_q$} [bl] at 0 5
\endlabellist
\centering
\includegraphics[scale=.9]{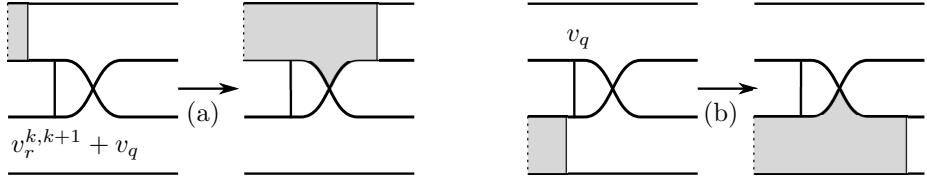}
\caption{Extending an $(\aug, \aug', H)$-half-disk past a degree $0$ crossing so as to have a convex corner at the crossing.}
\label{f:degree0}
\end{figure}

Note that $\mathcal{H}(x_{r+1}, [k, k+1])$ is empty. Suppose one of $i = k+1$ or $j=k$ hold. Half-disks in $\mathcal{H}(x_{r}, [\rho(i), \rho(j)])$ may be smoothly extended past $q$ as in Figure~\ref{f:simplecrossing} (a) and (b). Therefore, there exists an injection from $\mathcal{H}(x_{r}, [\rho(i), \rho(j)])$ to $\mathcal{H}(x_{r+1}, [i, j])$. However, there may be half-disks in $\mathcal{H}(x_{r+1}, [i, j])$ that have a convex corner at $q$. If $j=k$ (resp. $i=k+1$) and $q$ is marked by $\MCS'$ (resp. $\MCS$), then a half-disk in $\mathcal{H}(x_{r+1}, [i, j])$ can have a convex corner at $q$; see Figure~\ref{f:degree0} (a) and (b) respectively. Such half-disks are in bijection with half-disks in $\mathcal{H}(x_{r}, [i, j])$, as can be seen in Figure~\ref{f:degree0} (a) and (b), and by Property~\ref{prop:count} (b), are counted by $v_r^{i,j}$. Since $v_r^{k,k+1}$ is $1$ if and only if $\MCS'$ and $\MCS$ differ at $q$, $q$ is marked by $\MCS'$ (resp. $\MCS$) if and only if $v_r^{k,k+1}+v_q$ is $1$ (resp. $v_q$ is $1$). Therefore, if $j=k$ (resp. $i=k+1$), then the mod 2 count of half-disks in $\mathcal{H}(x_{r+1}, [i, j])$ with a convex corner at $q$ is $(v_r^{k,k+1}+v_q) \cdot v_r^{i,k}$ (resp. $v_q \cdot v_r^{k+1,j}$). In summary,
\begin{equation}
\label{eq:crossing2}
\# \mathcal{H}(x_{r+1}, [i, j]) = \left\{ \begin{array}{rl}
 v_{r}^{i,k+1} + (v_r^{k,k+1}+v_q) \cdot v_r^{i,k}  &\mbox{ if $j = k$} \\
 v_{r}^{k,j} + v_q \cdot v_r^{k+1,j} &\mbox{ if $i = k+1$} \\ 
 0 &\mbox{ if $i=k$ and $j=k+1$} \\
 v_{r}^{\rho(i),\rho(j)}  &\mbox{ otherwise.} 
       \end{array} \right. 
\end{equation}       

We now define the sequence of MCS moves that create $\MCS_{r+1}$ from $\MCS_r$ and prove Property~\ref{prop:count} holds for $\MCS_{r+1}$. We move each handleslide $v_r^{i,j}$ of $V_r$ past $q$ beginning with the right-most handleslide in $V_r$. If $i \geq k$ and $j \neq k+1$ hold, use MCS moves (2) - (9) to move $v_r^{i,j}$ past $v_q$, if $v_q$ is $1$, and then past the crossing $q$. If $v_q$ is $1$ and $i=k+1$, a new handleslide with endpoints on strands $k$ and $j$ is created when MCS move (4) is used to move $v_r^{i,j}$ past $v_q$. Move this handleslide to the right of $q$ as well. It is not possible to move $v_r^{k,k+1}$ past $q$ and so, for now, we simply leave $v_r^{k,k+1}$ to the left of $q$. If $i < k$ holds, use MCS moves (2) - (9) to move $v_r^{i,j}$ past $v_r^{k,k+1}$, if $v_r^{k,k+1}$ is $1$, then past $v_q$, if $v_q$ is $1$, and then past the crossing $q$. If $v_q$ is $1$ or $v_r^{k,k+1}$ is $1$, and $j=k$, a new handleslide with endpoints on strands $i$ and $k+1$ is created when MCS move (4) is used to move $v_r^{i,j}$ past $v_q$ or $v_r^{k,k+1}$. Move this handleslide to the right of $q$ as well. Once all $v_r^{i,j}$, except $v_r^{k,k+1}$, have been moved past $q$, use MCS moves (5) and (1) to order the collection of handleslides just to the right of $q$ and  remove pairs of handleslides that have the same endpoints. The resulting collection is $V_{r+1}$. From our work above, we have:
\begin{equation}
\label{eq:crossing}
v_{r+1}^{i,j} = \left\{ \begin{array}{rl}
 v_{r}^{i,k+1} + (v_r^{k,k+1}+v_q) \cdot v_r^{i,k}  &\mbox{ if $j = k$} \\
 v_{r}^{k,j} + v_q \cdot v_r^{k+1,j} &\mbox{ if $i = k+1$} \\ 
 0 &\mbox{ if $i=k$ and $j=k+1$} \\
 v_{r}^{\rho(i),\rho(j)}  &\mbox{ otherwise.} 
       \end{array} \right.
\end{equation}

\begin{figure}[t]
\centering
\includegraphics[scale=.9]{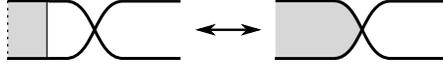}
\caption{The bijection between $(\aug, \aug', H)$-half-disks originating at $\{x_r\} \times [k,k+1]$ and $(\aug, \aug', H)$-half-disks originating at a crossing $q$ between strands $k$ and $k+1$.}
\label{f:crossingstrands}
\end{figure}

Use MCS move (1) to remove both $v_r^{k,k+1}$ and $v_q$, in the case that they both exist. The resulting MCS is $\MCS_{r+1}$.  By Property~\ref{prop:count}, $v_r^{k,k+1}$ is $1$ if and only if there is an odd number of $(\aug, \aug', H)$-half-disks originating at $\{x_r\} \times [k,k+1]$. There is a bijection between such disks and the $(\aug, \aug', H)$-admissible disks originating at $q$; see Figure~\ref{f:crossingstrands}. Therefore, by Proposition~\ref{prop:admissible-disks}, $v_r^{k,k+1}$ is $1$ if and only if $\MCS$ and $\MCS'$ differ at $q$. Therefore, Property~\ref{prop:count} (a) holds for $\MCS_{r+1}$. Finally, Equations (\ref{eq:crossing}) and (\ref{eq:crossing2}) imply Property~\ref{prop:count} (b) holds for $\MCS_{r+1}$. 

\textbf{Crossing $p_i$ where $1 < i \leq m$:} Suppose $p_i$ is a degree $-1$ crossing between strands $k$ and $k+1$ and $H(p_i) = 1$ holds. Suppose $i \neq k+1$ and $j \neq k$. Half-disks in $\mathcal{H}(x_{r+1}, [i, j])$ cannot have a convex corner at $p_i$ and, therefore, there is a bijection from $\mathcal{H}(x_{r}, [i,  j])$ to $\mathcal{H}(x_{r+1}, [\rho(i), \rho(j)])$; see, for example, Figure~\ref{f:simplecrossing} (c) - (f). Since Property~\ref{prop:count} (b) holds for $\MCS_r$, $\# \mathcal{H}(x_{r+1}, [i, j]) = v_r^{\rho(i), \rho(j)}$ holds. 

\begin{figure}[t]
\labellist
\small\hair 2pt
\pinlabel {(a)} [tl] at 159 143
\pinlabel {(b)} [tl] at 17 100
\pinlabel {(c)} [tl] at 283 100
\pinlabel {(d)} [tl] at 91 32
\pinlabel {(e)} [tl] at 224 32
\endlabellist
\centering
\includegraphics[scale=.9]{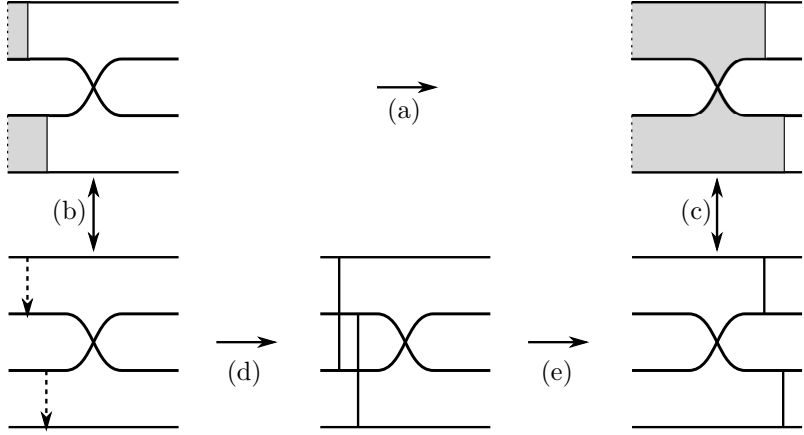}
\caption{The correspondence between $(\aug, \aug', H)$-half-disks with a convex corner at the degree $-1$ crossing in the figure and handleslides introduced by MCS move (13) to the left of the crossing.}
\label{f:crossingcase}
\end{figure}

Note that $\mathcal{H}(x_{r+1}, [k, k+1])$ is empty. Suppose one of $i = k+1$ or $j=k$ hold. Half-disks in $\mathcal{H}(x_{r}, [\rho(i), \rho(j)])$ may be smoothly extended past $p_i$ as in Figure~\ref{f:simplecrossing} (a) and (b). Therefore, there exists an injection from $\mathcal{H}(x_{r}, [\rho(i), \rho(j)])$ to $\mathcal{H}(x_{r+1}, [i, j])$. However, there may be half-disks in $\mathcal{H}(x_{r+1}, [i, j])$ that have a convex corner at $q$. In the case that $j= k$ (resp. $i=k+1$), such disks correspond to $\aug$-half-disks (resp. $\aug'$-half-disks) that have been extended past $p_i$ so as to have a convex corner at $p_i$; see Figure~\ref{f:crossingcase} (a) - (e). Let $(C, d)$ (resp. $(C', d')$) be the chain complex of $\MCS$ (resp. $\MCS'$) just to the left of $p_i$. Property~\ref{prop:count} (a) implies that, in $\MCS_r$, $(C, d)$ (resp. $(C',d')$) is the chain complex to the immediate right (resp. left) of $V_r$.  By Lemma~\ref{l:aug-half-disks}, the mod 2 count of such half-disks is $\langle d e_i, e_j \rangle$ and $\langle d' e_i, e_j \rangle$ respectively. Since Property~\ref{prop:count} (b) holds for $\MCS_r$, we may summarize the work of the previous two paragraphs as follows,
\begin{equation}
\label{eq:crossinghalfdisks1}
\# \mathcal{H}(x_{r+1}, [i, j]) = \left\{ \begin{array}{rl}
 v_r^{\rho(i), \rho(j)} + \langle d e_i, e_j \rangle  &\mbox{ if $j = k$} \\
 v_r^{\rho(i), \rho(j)} + \langle d' e_i, e_j \rangle &\mbox{ if $i = k+1$} \\ 
 0 &\mbox{ if $i=k$ and $j=k+1$} \\
 v_{r}^{\rho(i),\rho(j)}  &\mbox{ otherwise.} 
       \end{array} \right. 
\end{equation}       

\begin{figure}[t]
\labellist
\small\hair 2pt
\pinlabel {(a)} [tl] at 119 217
\pinlabel {(b)} [tl] at 276 217
\pinlabel {(c)} [tl] at 340 150
\pinlabel {(d)} [tl] at 284 56
\pinlabel {(e)} [tl] at 141 56
\pinlabel {$V$} [tl] at 55 145
\pinlabel {$V'$} [tl] at 12 191
\pinlabel {$p_i$} [tl] at 87 225
\endlabellist
\centering
\includegraphics[scale=.9]{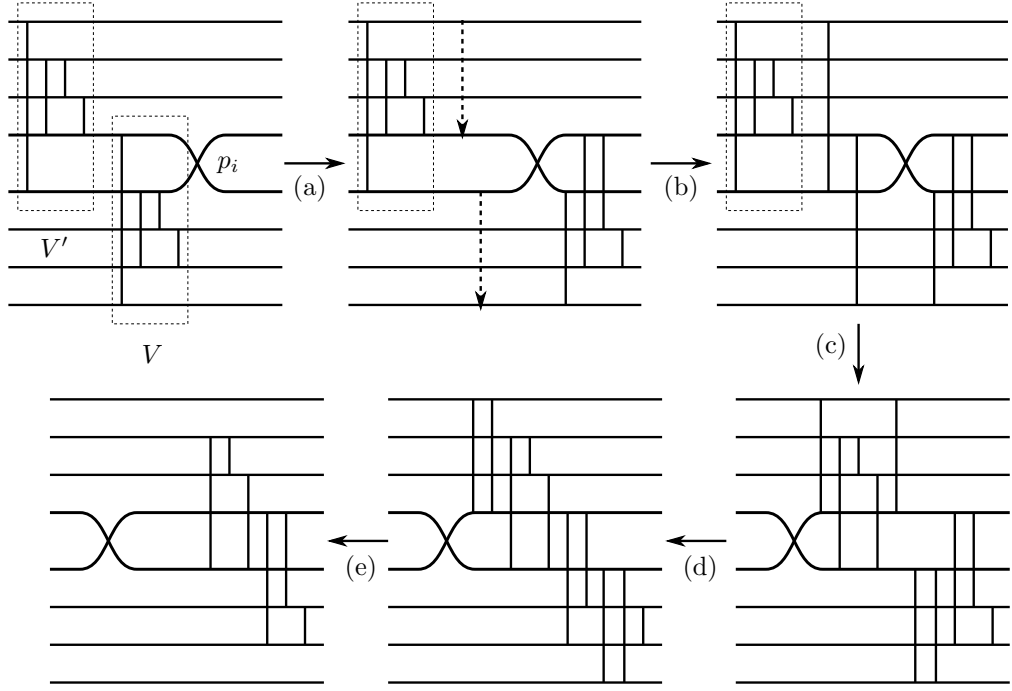}
\caption{The sequence of MCS moves at a crossing $p_i$ where $|p_i| = -1$ and $H(p_i) = 1$ both hold.}
\label{f:negcrossing}
\end{figure}

We now define the sequence of MCS moves that create $\MCS_{r+1}$ from $\MCS_r$ and prove Property~\ref{prop:count} holds for $\MCS_{r+1}$. Let $V \subset V_r$ (resp. $V' \subset V_r$) be the handleslides $v_r^{i,j}$ in $V_r$ satisfying $i \geq k$ (resp. $i < k$). Since $V_r$ is ordered, $V$ is right of $V'$; see Figure~\ref{f:negcrossing}. Let $(\bar{C},  \bar{d})$ be the chain complex of $\MCS_r$ between $V'$ and $V$. Use MCS moves (7) and (9) to move the handleslides in $V$ past $p_i$; see Figure~\ref{f:negcrossing} (a). Since $p_i$ has degree $-1$, $v_r^{k, k+1}$ is $0$ and $\Maslov(k) = \Maslov(k+1)-1$ holds. Therefore, strands $k$ and $k+1$ satisfy the conditions of MCS move (13). Use MCS move (13) to introduce new handleslides between $V'$ and $p_i$; see Figure~\ref{f:negcrossing} (b). MCS move (13) introduces a handleslide with endpoints $i$ and $j$ if and only if either $j = k+1$ and $\langle \bar{d} e_i, e_k \rangle$ is $1$, or $i=k$ and $\langle \bar{d} e_{k+1}, e_j \rangle$ is $1$. Recall that $(C, d)$ (resp. $(C', d')$) is the chain complex of $\MCS_r$ to the immediate right (resp. left) of $V_r$. Since $V_r$ is ordered, the handleslides between $(C, d)$ and $(\bar{C}, \bar{d})$ have upper endpoints on $k, \hdots, s_r$ and the handleslides between $(C', d')$ and $(\bar{C}, \bar{d})$ have upper endpoints on $1, \hdots, k-1$. Because of the ordering of handleslides in $V_r$, the coefficient $\langle \bar{d} e_i, e_k \rangle$ (resp. $\langle \bar{d} e_{k+1}, e_j \rangle$) is unaffected by handleslides in $V$ (resp. $V'$). As a consequence $\langle \bar{d} e_i, e_k \rangle = \langle d e_i, e_k \rangle$ holds for all $i < k$ and $\langle \bar{d} e_{k+1}, e_j \rangle = \langle d' e_{k+1}, e_{j} \rangle$ holds for all $k+1 < j$. Therefore, MCS move (13) introduces a handleslide with endpoints $i$ and $j$ if and only if either $j = k+1$ and $\langle d e_i, e_k \rangle$ is $1$, or $i=k$ and $\langle d' e_{k+1}, e_j \rangle$ is $1$. Move the handleslides created by MCS move (13) and the handleslides in $V'$ past $p_i$ using MCS moves (7) - (9); see Figure~\ref{f:negcrossing} (c). Use MCS moves (1), (3), (5) and (6) to order the collection of handleslides to the right of $p_i$ and remove pairs of handleslides with identical endpoints; see Figure~\ref{f:negcrossing} (d) and (e). In particular, this can be done without creating any new handleslides. The resulting ordered collection of handleslides is $V_{r+1}$ and the MCS is $\MCS_{r+1}$. Since the only new handleslides created were those created by the single application of MCS move (13), 
\begin{equation}
\label{eq:crossinghalfdisks2}
v_{r+1}^{i,j} = \left\{ \begin{array}{rl}
 v_r^{\rho(i), \rho(j)} + \langle d e_i, e_j \rangle  &\mbox{ if $j = k$} \\
 v_r^{\rho(i), \rho(j)} + \langle d' e_i, e_j \rangle &\mbox{ if $i = k+1$} \\ 
 0 &\mbox{ if $i=k$ and $j=k+1$} \\
 v_{r}^{\rho(i),\rho(j)}  &\mbox{ otherwise.} 
       \end{array} \right. 
\end{equation}       
Equations (\ref{eq:crossinghalfdisks1}) and (\ref{eq:crossinghalfdisks2}) imply Property~\ref{prop:count} (b) holds for $\MCS_{r+1}$. Finally, Property~\ref{prop:count} (a) holds for $\MCS_r$ and $|q| \neq 0$ implies $q$ is not marked by either $\MCS_{r+1}$ or $\MCS'$. Therefore, Property~\ref{prop:count} (a) holds for $\MCS_{r+1}$

\textbf{Left cusp:} Suppose $q$ is a left cusp. Number the strands of $\front$, from top to bottom, $1, \hdots, s_{r}$ (resp. $1, \hdots, s_{r+1}$) just to the left (resp. right) of $q$. Define $\tau: \{1,\ldots, s_{r}\} \rightarrow \{1, \ldots, s_{r+1}\}$ by $\tau(i) =\left\{ \begin{array}{cr} i,  & \mbox{if $i<k$} \\ i+2 & \mbox{if $i\geq k$.}  \end{array} \right.$  (Note that $s_{r} = s_{r+1} -2$.) For any $1 \leq i' < j' \leq s_{r}$, there is a bijection between $\mathcal{H}(x_{r+1}, [\tau(i'), \tau(j')])$ and $\mathcal{H}(x_{r}, [i', j'])$; see, for example, Figure~\ref{f:leftcuspcase} (a) and (b). If $\{i, j\} \cap \{k, k+1\}$ is non-empty, then $\mathcal{H}(x_{r+1}, [i, j])$ is empty. Therefore, since Property~\ref{prop:count} (b) holds for $\MCS_{r}$, 
\begin{equation}
\label{eq:leftcuspH}
\# \mathcal{H}(x_{r+1}, [i, j]) = \left\{ \begin{array}{rl}
 v_r^{\tau^{-1}(i), \tau^{-1}(j)} &\mbox{ if $\{i, j \} \cap \{k, k+1\} = \emptyset$} \\
  0 &\mbox{ otherwise.}
       \end{array} \right.
\end{equation}

Use MCS moves (11) and (12) to move each handleslide in $V_r$ past $q$. The resulting collection $V_{r+1}$ is ordered and the resulting MCS is $\MCS_{r+1}$. The endpoints of a handleslide remain on the same strands of $\front$ as it is moved past $q$. Therefore, we have
\begin{equation}
\label{eq:leftcusphandleslide}
v_{r+1}^{i, j} = \left\{ \begin{array}{rl}
 v_r^{\tau^{-1}(i), \tau^{-1}(j)} &\mbox{ if $\{i, j \} \cap \{k, k+1\} = \emptyset$} \\
  0 &\mbox{ otherwise.}
       \end{array} \right.
\end{equation}

Equations (\ref{eq:leftcuspH}) and (\ref{eq:leftcusphandleslide}) imply Property~\ref{prop:count} (b) holds for $\MCS_{r+1}$. Since $q$ is not a crossing and Property~\ref{prop:count} (a) holds for $\MCS_r$, it must hold for $\MCS_{r+1}$ as well. 

\textbf{Right cusp:} Suppose $q$ is a right cusp between strands $k$ and $k+1$. Let $(C, d)$ (resp. $(C', d')$) be the chain complex of $\MCS$ (resp. $\MCS'$) just to the left of $q$. Property~\ref{prop:count} (a) implies that, in $\MCS_r$, $(C, d)$ (resp. $(C',d')$) is the chain complex to the immediate right (resp. left) of $V_r$. Number the strands of $\front$, from top to bottom, $1, \hdots, s_{r+1}$ (resp. $1, \hdots, s_r$) just to the right (resp. left) of $q$. Define $\pi: \{1,\ldots, s_{r+1}\} \rightarrow \{1, \ldots, s_r\}$ by $\pi(i) =\left\{ \begin{array}{cr} i,  & \mbox{if $i<k$} \\ i+2 & \mbox{if $i\geq k$.}  \end{array} \right.$  (Note that $s_{r} = s_{r+1} +2$.) 

\begin{figure}[t]
\labellist
\small\hair 2pt
\pinlabel {(a)} [tl] at 136 277
\pinlabel {(b)} [tl] at 13 244
\pinlabel {(c)} [tl] at 231 244
\pinlabel {(d)} [tl] at 76 186
\pinlabel {(e)} [tl] at 189 186
\pinlabel {(f)} [tl] at 76 106
\pinlabel {(g)} [tl] at 189 106
\pinlabel {(h)} [tl] at 76 42
\pinlabel {(i)} [tl] at 190 42
\pinlabel {$l$} [tr] at 0 225
\pinlabel {$k$} [tr] at 0 209
\pinlabel {$k+1$} [tr] at 0 184
\pinlabel {$j$} [tr] at 0 169
\pinlabel {$h$} [tr] at 134 233
\endlabellist
\centering
\includegraphics[scale=.9]{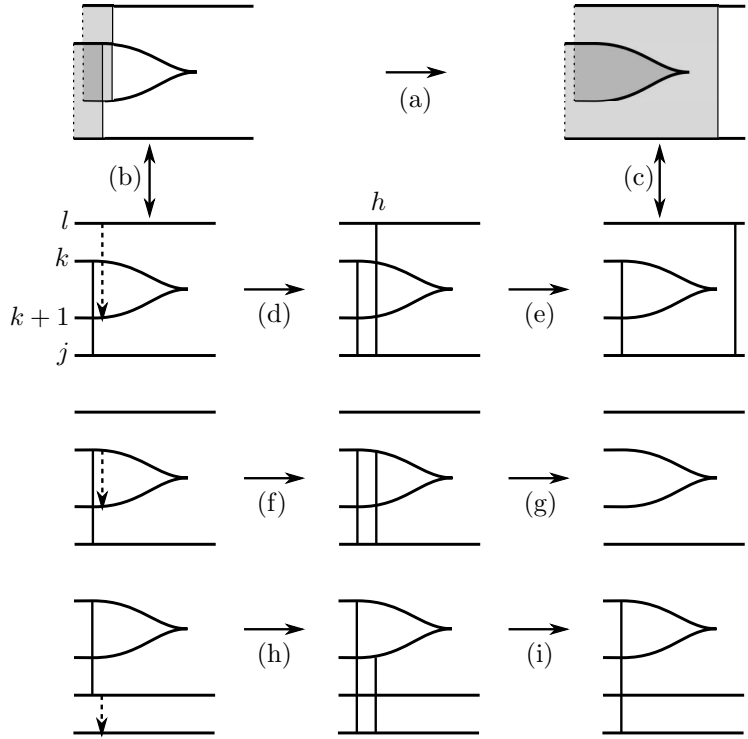}
\caption{(a) - (e): The correspondence between Type 1 $(\aug, \aug', H)$-half-disks whose boundary intersects the right cusp in the figure and handleslides introduced by MCS move (13) to the left of the crossing. (f), (g): A handleslide introduced by MCS move (13) that is removed, along with $v_r^{k,j}$, by an MCS (1) move. (h), (i): A handleslide introduced by MCS move (13) that is removed by an MCS (10) move.}
\label{f:Type1}
\end{figure}

\begin{figure}[t]
\labellist
\small\hair 2pt
\pinlabel {(a)} [tl] at 136 119
\pinlabel {(b)} [tl] at 13 86
\pinlabel {(c)} [tl] at 231 86
\pinlabel {(d)} [tl] at 76 27
\pinlabel {(e)} [tl] at 189 27
\pinlabel {$i$} [tr] at 0 64
\pinlabel {$k$} [tr] at 0 46
\pinlabel {$k+1$} [tr] at 0 23
\pinlabel {$l$} [tr] at 0 6
\pinlabel {$h$} [tr] at 134 72
\endlabellist
\centering
\includegraphics[scale=.9]{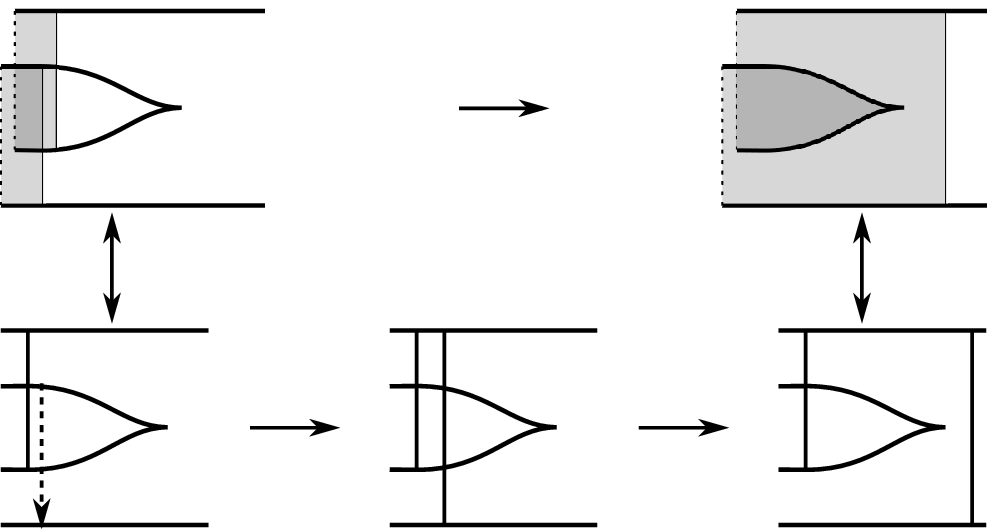}
\caption{(a) - (e): The correspondence between Type 2 $(\aug, \aug', H)$-half-disks whose boundary intersects the right cusp in the figure and handleslides introduced by MCS move (13) to the left of the crossing.}
\label{f:Type2}
\end{figure}

If $j < k$ or $i \geq k$, then 
\begin{equation}
\label{eq:Type0}
\# \mathcal{H}(x_{r+1}, [i, j]) = v_r^{\pi(i),\pi(j)}
\end{equation}
holds, since Property~\ref{prop:count} (b) holds for $\MCS_r$ and there is a bijection from $\mathcal{H}(x_{r}, [\pi(i), \pi(j)])$ to $\mathcal{H}(x_{r+1}, [i, j])$; see Figure~\ref{f:cuspcase} (a) and (b). 

When $j \geq k$ and $i < k$, we claim that there is a bijection 
\begin{align}
\label{eq:right-cusp}
\mathcal{H}(x_{r+1}, [i, j]) &\cong \mathcal{H}(x_{r}, [\pi(i), \pi(j)]) \\ \nonumber
&\cup \left(\mathcal{G}^\epsilon(x_{r}, [\pi(i), k+1]) \times \mathcal{H}(x_r, [k,\pi(j)]) \right) \\ \nonumber
&\cup \left(\mathcal{H}(x_{r}, [\pi(i), k+1]) \times \mathcal{G}^{\epsilon'}(x_r, [k,\pi(j)]) \right).
\end{align}
Suppose $j \geq k$ and $i < k$. Half-disks in $\mathcal{H}(x_{r}, [\pi(i), \pi(j)])$ may be smoothly extended past $q$ as in Figure~\ref{f:cuspcase} (c). Therefore, there exists an injection from $\mathcal{H}(x_{r}, [\pi(i), \pi(j)])$ to $\mathcal{H}(x_{r+1}, [i, j])$. However, there may be half-disks in $\mathcal{H}(x_{r+1}, [i, j])$ whose boundary intersects the cusp point; see Figure~\ref{f:Type1} (a) and Figure~\ref{f:Type2} (a). 

We divide half-disks whose boundary intersects $q$ into two types as follows.  Any such half-disk has one convex corner at a degree $-1$ crossing, which we denote $p$. Trace the boundary of such a half-disk counter-clockwise beginning at the vertical line $\{x_{r+1}\}\times [i,j]$. In a Type 1 (resp. Type 2) half-disk, $p$ appears after (resp. before) $q$. A Type 1 (resp. Type 2) half-disk can be uniquely decomposed into an $(\aug, \aug', H)$-half-disk and an $\aug$-half-disk (resp. $\aug'$-half-disk) as in Figure~\ref{f:Type1} (a) and (b) (resp. Figure~\ref{f:Type2} (a) and (b)). Therefore, the set in the second (resp. third) line of Equation (\ref{eq:right-cusp}) is in bijection with Type 1 (resp. Type 2) half-disks. Since Property~\ref{prop:count} (b) holds for $\MCS_r$ and Lemma~\ref{l:aug-half-disks} holds for both $(C, d)$ and $(C',d')$, the mod 2 count of Type 1 half-disks is $\langle d e_{\pi(i)}, e_{k+1} \rangle \cdot v_r^{k, \pi(j)}$ and the mod 2 count of Type 2 half-disks is $v_r^{\pi(i),k+1} \cdot \langle d' e_k, e_{\pi(j)} \rangle$. Therefore, for $j \geq k$ and $i < k$, we have the formula
\begin{align}
\label{eq:Type12}
\# \mathcal{H}(x_{r+1}, [i, j]) &=  v_r^{\pi(i),\pi(j)} + v_r^{\pi(i),k+1} \cdot \langle d' e_k, e_{\pi(j)} \rangle \\ \nonumber
&+  \langle d e_{\pi(i)}, e_{k+1} \rangle \cdot v_r^{k, \pi(j)}.
\end{align}

We now define the sequence of MCS moves that create $\MCS_{r+1}$ from $\MCS_r$ and prove Property~\ref{prop:count} holds for $\MCS_{r+1}$. We move the handleslides of $V_r$ past $q$ iteratively beginning with the right-most handleslide. Suppose $v_r^{i,j}$ is the right-most handleslide of $V_r$ that has yet to be moved past $q$. If $i > k+1$ or $j < k$, use MCS move (12) to move $v_r^{i,j}$ past $q$. If $i<k$ and $j> k+1$, use MCS move (11) to move $v_r^{i,j}$ past $q$. If $i = k+1$ or $j=k$, use MCS move (10) to remove $v_r^{i,j}$. Since $\Maslov(k) = \Maslov(k+1) + 1$ and a handleslide has endpoints on strands with the same Maslov potential, $v_r^{k, k+1}$ must be $0$. It remains to consider the two cases $i=k, j>k+1$ and $i<k, j=k+1$. 

Suppose $v_r^{i,j}$ is $v_r^{k,j}$ where $j > k+1$. Since $\Maslov(k) = \Maslov(j)$ and $\Maslov(k) = \Maslov(k+1) + 1$ both hold, $\Maslov(k+1) = \Maslov(j) - 1$ holds and, thus, strands $k+1$ and $j$ satisfy the conditions of MCS move (13). Use MCS move (13) to create new handleslide marks; see the arrow directed to the right in Figure~\ref{f:explosion}. Let $(\bar{C}, \bar{d})$ be the chain complex of $\MCS_r$ just to the right of $v_r^{k,j}$. The handleslides created are of three types. By Definition~\ref{defn:MCS} (4), $\langle \bar{d} e_{k}, e_{k+1} \rangle$ is $1$. Therefore, MCS move (13) introduces a handleslide with endpoints $k$ and $j$; see Figure~\ref{f:Type1} (f). Use MCS move (1) to remove this handleslide and $v_r^{k, j}$; see Figure~\ref{f:Type1} (g). For each $l$ such that $\langle \bar{d} e_j, e_l \rangle$ is $1$, MCS move (13) introduces a handleslide with endpoints $k+1$ and $l$; see Figure~\ref{f:Type1} (h). Use MCS move (10) to remove this handleslide; see Figure~\ref{f:Type1} (i). Suppose $l < k$ and $\langle \bar{d} e_l, e_{k+1} \rangle$ is $1$. The third type of handleslide introduced by MCS move (13) has endpoints $l$ and $j$; see Figure~\ref{f:Type1} (d). Let $h$ be this handleslide. Use MCS move (11) to move $h$ past $q$; see Figure~\ref{f:Type1} (e). Recall that $(C, d)$ is the chain complex of $\MCS_r$ to the immediate right of $V_r$. Since $V_r$ is ordered, the handleslides between $(\bar{C}, \bar{d})$ and $(C, d)$ have endpoints on strands $k+1, \hdots, s_p$. The coefficient $\langle \bar{d} e_l, e_{k+1} \rangle$ is unaffected by such handleslides and, thus, $\langle \bar{d} e_l, e_{k+1} \rangle = \langle d e_l, e_{k+1} \rangle$ holds. Therefore, $h$ exists if and only if $\langle d e_{l}, e_{k+1} \rangle \cdot v_r^{k, j}$ is $1$. As we noted earlier, $\langle d e_{l}, e_{k+1} \rangle \cdot v_r^{k, j}$ is $1$ if and only if the mod 2 count of Type 1 half-disks in $\mathcal{H}(x_{r+1}, [l, j+2])$ is $1$. Therefore, $h$ exists if and only if the mod 2 count of Type 1 half-disks in $\mathcal{H}(x_{r+1}, [l, j+2])$ is $1$.

Suppose $v_r^{i,j}$ is $v_r^{i, k+1}$ where $i < k$. Note that $\Maslov(i) = \Maslov(k) - 1$ holds and, thus, strands $i$ and $k$ satisfy the conditions of MCS move (13). Use MCS move (13) to create new handleslides to the immediate right of $v_r^{i, k+1}$. Suppose $l > k+1$ and $\langle \bar{d} e_k, e_{l} \rangle$ is $1$. MCS move (13) introduces a handleslide with endpoints $i$ and $l$; see Figure~\ref{f:Type2} (d). Let $h$ be this handleslide. Use MCS move (11) to move $h$ past $q$; see Figure~\ref{f:Type2} (e). Following an analogous argument as was used in the case of a Type 1 half-disk, $h$ exists if and only if the mod 2 count of Type 2 half-disks in $\mathcal{H}(x_{r+1}, [i, l+2])$ is $1$.
MCS move (13) also introduces handleslides analogous to those in Figure~\ref{f:Type1} (f) and (h), which are removed in same manner as was done in Figure~\ref{f:Type1} (g) and (i). 

Once we have applied the above algorithm to each handleslide in $V_r$, we are left with a collection of handleslides $V$ to the right of $q$. The ordering of $V_r$ ensures the only new handleslides were those introduced by applications of MCS move (13). Therefore, given $1 \leq i < j \leq s_{r+1}$, there may be up to 3 handleslides in $V$ with endpoints on $i$ and $j$; one counts $(\aug, \aug', H)$-half-disks extended past $q$ as in Figure~\ref{f:cuspcase}, one counts Type 1 half-disks as in Figure~\ref{f:Type1} (a) - (e), and the third counts Type 2 half-disks as in Figure~\ref{f:Type2} (a) - (e). Use MCS moves (1), (3), (5), and (6) to remove pairs of handleslides with identical endpoints and order $V$. In particular, $V$ can be ordered without creating new handleslides. The resulting ordered collection of handleslides is $V_r$ and the MCS is $\MCS_{r+1}$. If $j < k$ or $i \geq k$, then $$v_{r+1}^{i,j} = v_r^{\pi(i),\pi(j)}$$ holds and, if $j \geq k$ and $i < k$, then
$$v_{r+1}^{i,j} =  v_r^{\pi(i),\pi(j)} + v_r^{\pi(i),k+1} \cdot \langle d' e_k, e_{\pi(j)} \rangle  +  \langle d e_{\pi(i)}, e_{k+1} \rangle \cdot v_r^{k, \pi(j)}$$ holds. These equations, along with Equations (\ref{eq:Type0}) and (\ref{eq:Type12}), imply Property~\ref{prop:count} (b) holds for $\MCS_{r+1}$. Finally, since $q$ is not a crossing and Property~\ref{prop:count} (a) holds for $\MCS_r$, it must hold for $\MCS_{r+1}$ as well. 

This completes the construction of the MCSs $\MCS_0, \ldots, \MCS_s$.  
\end{proof}

In the following corollaries, $\front$ is the $\sigma$-generic front diagram of a Legendrian knot with rotation number $0$. Recall that an augmentation $\aug$ in $\mbox{Aug}(\front)$ has a corresponding A-form MCS $\MCS$ where, for a degree $0$ crossing $q$, $\aug(q) = 1$ holds if and only if $q$ is marked by $\MCS$. 

\begin{corollary}
\label{c:MCS-invt}
The count of MCS classes of a Legendrian knot is a Legendrian isotopy invariant. 
\end{corollary}

Corollary~\ref{c:MCS-invt} follows from the fact that the count of homotopy classes of augmentations is a Legendrian isotopy invariant and every Legendrian knot is Legendrian isotopic to a Legendrian knot with $\sigma$-generic front diagram by an arbitrarily small Legendrian isotopy. Corollary~\ref{c:MCS-invt} is stated and a proof is briefly sketched by Petya Pushkar in a letter to Dmitry Fuchs in 2000. The proposed proof investigates the effect of Legendrian Reidemeister moves on the number of MCS classes and is different from the approach in this article. 

Given the Chekanov-Eliashberg algebra $(\alg(\front), \df)$, the differential $\df^{\aug} : \alg(\front) \to \alg(\front)$ is $\phi^{\aug} \circ \df \circ (\phi^{\aug})^{-1}$ where $\phi^{\aug} : \alg(\front) \to \alg(\front)$ is the algebra map defined on generators by $\phi^{\aug}(q) = q + \aug(q)$. The group $\mbox{LCH}(\aug)$, called the \textbf{linearized contact homology} of $\aug$, is the homology of the chain complex $(A(\front), \df^{\aug}_1)$ where $\df^{\aug}_1(q)$ is the length 1 monomials of $\df^{\aug}(q)$. By \cite{Chekanov2002}, the set $\{ \mbox{LCH}(\aug) \}_{\aug \in \mbox{Aug}(\front)}$ is a Legendrian isotopy invariant, which we will call the \textbf{LCH invariant}. 

\begin{corollary}
\label{c:LCH}
If $\aug$ and $\aug'$ are homotopic as augmentations, then $\mbox{LCH}(\aug)$ and $\mbox{LCH}(\aug')$ are isomorphic as homology groups. Therefore, augmentation homotopy classes have well-defined linearized contact homology groups. 
\end{corollary}

\begin{proof}
We will apply two theorems from \cite{Henry2011}. In order to do so, the front diagram must be ``nearly plat''. A front diagram is \textbf{plat} if all left cusps have the same $x$-coordinate, all right cusps have the same $x$-coordinate, and no two crossings have the same $x$-coordinate. A front diagram is \textbf{nearly plat} if it is the result of perturbing a plat front diagram slightly so that no two cusps have the same $x$-coordinate. 

We now deduce the corollary in the case that  $\front$ is nearly plat. Suppose $\aug$ and $\aug'$ are homotopic. By Lemma~\ref{l:equiv}, the A-form MCSs $\MCS$ and $\MCS'$ corresponding to $\aug$ and $\aug'$ are equivalent as MCSs. In \cite{Henry2013}, differential graded algebras $(\alg_{\MCS}, d)$ and $(\alg_{\MCS'}, d')$ are assigned to $\MCS$ and $\MCS'$, respectively. The linear level of each algebra is a chain complex $(A_{\MCS}, d_1)$ and $(A_{\MCS'}, d'_1)$, respectively. By Theorem 5.5 of \cite{Henry2013}, $(A_{\MCS}, d_1)$ and $(A_{\MCS'}, d'_1)$ are isomorphic. By Theorem 7.3 of \cite{Henry2013}, $(A(\front), \df^{\aug}_1)$ is isomorphic to $(A_{\MCS}, d_1)$ and  $(A(\front), \df^{\aug'}_1)$ is isomorphic to $(A_{\MCS'}, d'_1)$. Therefore, $\mbox{LCH}(\aug)$ and $\mbox{LCH}(\aug')$ are isomorphic as homology groups. 

For the general case of a Chekanov-Eliashberg algebra $(\alg, \df)$ assigned to a front (or Lagrangian) diagram that is not nearly plat, we argue as follows. By \cite{Chekanov2002a}, Chekanov-Eliashberg algebras assigned to Legendrian isotopic Legendrian knots are stable tame isomorphic. Any Legendrian knot is Legendrian isotopic to a knot with nearly plat front diagram, therefore $(\alg, \df)$ is stable tame isomorphic to a DGA that satisfies the property stated in Corollary \ref{c:LCH}. We then verify that $(\alg, \df)$ also satisfies the corollary in two steps.

\begin{enumerate}
\item[Step 1.]  The corollary holds for a stabilization $(S(\alg), \df')$ of a DGA $(\alg, \df)$ if and only if it holds for $(\alg, \df)$. 

Here, $S(\alg)$ is obtained from $\alg$ by adding two generators $x$ and $y$ in successive degrees, and the differential satisfies $\df'|_{\alg} = \df$ and $\df'x = y$.  Restricting augmentations of $S(\alg)$ to $\alg$ provides a surjection from the set of augmentations of $S(\alg)$ to the set of augmentations of $\alg$, and this gives a well-defined bijection between homotopy classes of augmentations of $S(\alg)$ and $\alg$.  Moreover,  for any augmentation $\aug : S(\alg) \rightarrow \Z/2$, the linearized homology groups associated to $\aug$ and $\aug|_{\alg}$ are isomorphic, so Step 1 follows.   

\item[Step 2.] If $\varphi: (\alg_1, \df_1)\rightarrow (\alg_2, \df_2)$  is an isomorphism of DGAs,  then the corollary holds for $(\alg_1, \df_1)$ if and only if it holds for $(\alg_2, \df_2)$.

To see this, observe that $\epsilon_2 \mapsto \epsilon_1 \circ \varphi$ gives a bijection from augmentations of $(\alg_2, \df_2)$ to augmentations of $(\alg_1, \df_1)$ that preserves homotopy classes and linearized homology groups. 
\end{enumerate}

\end{proof}

Corollary~\ref{c:LCH} provides a means for strengthening the LCH invariant. The set $\{ \mbox{LCH}(\aug) \}_{\aug \in \mbox{Aug}(\front)}$, along with a count of the number of augmentation homotopy classes associated with each group, is a Legendrian isotopy invariant. The authors are currently unaware of an example where this refinement is able to distinguish knots that are not already distinguished by the LCH invariant taken without regard to multiplicity.

\begin{corollary}
If $\aug$ and $\aug'$ are homotopic, then $\aug$ and $\aug'$ are mapped to the same graded normal ruling by the many-to-one map from augmentations to graded normal rulings defined in \cite{Ng2006}.
\end{corollary}

\begin{proof}
Suppose $\aug$ and $\aug'$ are homotopic. By Lemma~\ref{l:equiv}, the A-form MCSs $\MCS$ and $\MCS'$ corresponding to $\aug$ and $\aug'$ are equivalent. By Lemma 3.14 of \cite{Henry2011}, every MCS determines a graded normal ruling. By Proposition 3.15 of \cite{Henry2011}, equivalent MCSs determine the same graded normal ruling. Therefore, $\MCS$ and $\MCS'$ determine the same graded normal ruling. In \cite{Ng2006}, there is an algorithmically defined many-to-one map $\Omega$ from $\mbox{Aug}(\front)$ to the set of graded normal rulings of $\front$. In the case of an augmentation $\aug$ and its corresponding A-form MCS $\MCS$, $\Omega(\aug)$ is the same as the graded normal ruling determined by $\MCS$ in Lemma 3.14 of \cite{Henry2011}. Therefore, $\Omega$ maps $\aug$ and $\aug'$ to the same graded normal ruling. 
\end{proof}

\addcontentsline{toc}{section}{Bibliography} 
\bibliographystyle{amsplain}
\bibliography{Bibliography}

\end{document}